\documentclass{article}
\usepackage{geometry}[top=15truemm, bottom=15truemm, right=10truemm, left=10truemm]
\usepackage{amsmath,amssymb,amsthm}
\usepackage{amsmath,amssymb,amsthm}
\usepackage[all]{xy}
\usepackage{graphicx}
\usepackage{mathrsfs}
\usepackage{lineno}

\newcommand{\intwitt}{W_{O_K}^a(O_{\bar{K}})}

\newcommand{\C}{\mathscr{C}}

\newcommand{\modular}{\mathfrak{F}}

\newcommand{\diag}{\mathrm{diag}}
\renewcommand{\exp}{{\bf e}}
\newcommand{\proj}{{\bf P}}
\newcommand{\tr}{\mathrm{Tr}}

\newcommand{\idp}{\mathfrak{p}}

\newcommand{\idf}{\mathfrak{f}}

\newcommand{\ida}{\mathfrak{a}}
\newcommand{\idb}{\mathfrak{b}}
\newcommand{\idc}{\mathfrak{c}}

\newcommand{\ratf}{\mathbb{Q}}
\newcommand{\comp}{\mathbb{C}}

\newcommand{\realf}{\mathbb{R}}
\newcommand{\upper}{\mathfrak{H}}
\newcommand{\Ad}{\mathbb{A}}

\newcommand{\F}{\mathrm{F}}

\newcommand{\End}{\mathrm{End}}
\newcommand{\Auto}{\mathrm{Aut}}
\newcommand{\imm}{\mathrm{Im}}
\newcommand{\Hom}{\mathrm{Hom}}

\newcommand{\integer}{\mathbb{Z}}

\newcommand{\idd}{\mathfrak{d}}

\theoremstyle{plain}
\newtheorem{thm}{Theorem}[subsection]
\newtheorem{lem}[thm]{Lemma}
\newtheorem{prop}[thm]{Proposition}
\newtheorem{cor}[thm]{Corollary}

\theoremstyle{definition}
\newtheorem{defn}[thm]{Definition}

\newtheorem{rem}[thm]{Remark}

\numberwithin{equation}{section}
\title{Semi-galois Categories IV:\\ A deformed reciprocity law for Siegel modular functions}
\author{Takeo Uramoto\\ Institute of Mathematics for Industry, Kyushu University}
\date{}
\begin{document}
\maketitle
\begin{abstract}
\noindent
This paper is a sequel to our previous work \cite{Uramoto20}, where we proved the \emph{``modularity theorem'' for algebraic Witt vectors} over imaginary quadratic fields. This theorem states that, in the case of imaginary quadratic fields $K$, the algebraic Witt vectors over $K$ are precisely those generated by the \emph{modular vectors} whose components are given by special values of deformation family of Fricke modular functions; arithmetically, this theorem implies certain congruences between special values of modular functions that are not necessarily galois conjugate. In order to take a closer look at this modularity theorem, the current paper extends it to the case of CM fields. The main results include (i) a construction of algebraic Witt vectors from special values of deformation family of Siegel modular functions on Siegel upper-half space $\upper_g$ given by ratios of theta functions, and (ii) a galois-theoretic characterization of which algebraic Witt vectors arise in this modular way, intending to exemplify a general galois-correspondence result which is also proved in this paper. 


\end{abstract}

\section{Introduction}
\label{s1}
The goal of this paper is to extend our previous work on the \emph{modularity theorem for algebraic Witt vectors} over imaginary quadratic fields (Theorem 4.3.1, \cite{Uramoto20}) to the case of CM fields in particular. This theorem states that, if $K$ is an imaginary quadratic field, the $K$-algebra $E_K = K \otimes \intwitt$ of \emph{algebraic Witt vectors} over $K$ is identical to the $K$-algebra $M_K= K[\widehat{f}_a \mid a \in \ratf^2/\integer^2]$ generated by the \emph{modular vectors} $\widehat{f}_a$, whose components are given by special values of Fricke's modular functions $f_a: \upper \rightarrow \comp$ ($a \in \ratf^2/\integer^2$), cf.\ \S \ref{s2}. This theorem connects  two objects (algebraic/analytic) of different nature; and technically this theorem is proved using \emph{Shimura's reciprocily law}, which relates special values of \emph{galois-conjugate} modular functions (\S 6.8 \cite{Shimura}). But compared to Shimura's reciprocity law, our modularity theorem is characterized by the fact that it relates, by congruences, special values of modular functions living at different levels that are \emph{not necessarily galois conjugate}; in this sense, our modularity theorem strengthens/deforms Shimura's reciprocity law in a way analogous to \emph{Hida theory}, e.g.\ mod-$p$ congruences between Fourier coefficients of non-galois-conjugate cusp forms \cite{Mazur}. In view of that Shimura's reciprocity law is now known for more general Shimura varieties \cite{Milne, Milne_Shih}, our modularity theorem would eventually be a facet of some more general theme. To pursue this prospect, the current paper provides further results of that general direction, working on the case of Siegel modular functions.

Our proof of the modularity theorem $E_K = M_K$ consisted of three major steps; the extension to CM field in this paper will basically follow this line. The first step is to see that we can identify the algebraic Witt vectors $\xi \in E_K$ with locally constant, $K^{ab}$-valued, $G_K$-equivariant functions on the \emph{profinite Deligne-Ribet monoid} $DR_K$ (cf.\ Corollary 3.1.6 \cite{Uramoto20}), formally, $E_K \simeq \Hom_{G_K}(DR_K, K^{ab})$; this identification is a consequence of the fact that the algebraic Witt vectors $\xi \in E_K$ generate all \emph{galois objects} \cite{Uramoto17} of the semi-galois category $(\C_K, \F_K)$ of certain $\Lambda$-rings whose fundamental monoid $\pi_1(\C_K, \F_K)$ was known isomorphic to $DR_K$ \cite{Borger_Smit11}; this identification itself holds for arbitrary number fields $K$ and thus is used in the current paper too. The second step is to construct algebraic Witt vectors using special values of the family $\{f_a\}$ of Fricke's modular functions $f_a: \upper \rightarrow \comp$ indexed by $a \in \ratf^2/\integer^2$; thanks to the first step, the second step is reduced to constructing locally constant, $K^{ab}$-valued, $G_K$-equivariant functions on $DR_K$ from these $f_a$'s, for which the adelic description of the Deligne-Ribet monoid $DR_K$ by Yalkinoglu \cite{Yalkinoglu13}, $DR_K \simeq \widehat{O}_K \times_{\widehat{O}_K^\times} (\Ad_{K,f}^\times/K^\times)$, plays an auxiliary role. The third step is to prove that these algebraic Witt vectors constructed from $f_a$'s (namely the modular vectors $\widehat{f}_a \in E_K$) actually generate $E_K$ over $K$; this last part is our key step, in which we utilized the \emph{galois-orbit decompositions} of finite Deligne-Ribet monoids $DR_\idf$ for conductors $\idf \in I_K$ given as $DR_\idf = \coprod_{\idd \mid \idf} C_{\idf/\idd}$ \cite{Yalkinoglu13} in order to separate components of modular vectors (= special values of Fricke functions $f_a$'s that may live at different levels), reducing it to Boolean-valued problems of judging prime ramifications (cf.\ the proof of Lemma 4.3.2 \cite{Uramoto20} in particular). 

As mentioned above too, the current work also follows basically this line; but in the case of CM fields, major differences occur in the second and third steps. First, we shall replace Fricke modular functions $f_a: \upper \rightarrow \comp$ with certain \emph{ratios $\theta^{k/l}_a$ of theta functions}, which naturally define deformation families of modular functions on the Siegel upper-half space $\upper_g$; nevertheless, the basic idea behind these two families of modular functions is essentially the same, in that both families of functions (Fricke functions and ratios of theta functions) are those equally appearing in the construction of projective embedding of complex tori (resp.\ of dimension one and higher dimension) \cite{Shimura, Mumford}, which will be a central step in our construction of modular vectors; essentially, their values (= components of modular vectors) both represent the coordinates of torsion points on elliptic curves and abelian varieties respectively. Our major subject in the second step is to give the right setting of necessary constructions and a proof that the ratios of theta functions naturally produce locally constant $K^{ab}$-valued $G_K$-equivariant functions on $DR_K$, hence, algebraic Witt vectors (i.e.\ \emph{modular vectors}) over CM fields $K$ (cf.\ \S \ref{s3}). 

Second, unlike the case of imaginary quadratic fields, these modular vectors do not generate all algebraic Witt vectors over CM fields; thus it should be of natural concern to characterize which algebraic Witt vectors are generated by our modular vectors. The subject of \S \ref{s4} is to discuss this problem, where we show that an algebraic Witt vector $\xi \in E_K$ is generated by the modular vectors if and only if it is periodic with respect to a certain monoid congruence $\sim_N$ of finite index on $DR_K$ for some $N \geq 1$, where the congruences $\sim_N$ are given purely in an adelic term. Technically speaking, this result is a major result of the current paper and intended to exemplify a general \emph{galois-theoretic correspondence} (mentioned in Remark 3.1.8 \cite{Uramoto20} without proof) between (i) $\Lambda$-subalgebras of $E_K$, (ii) profinite quotients of $DR_K$, and (iii) semi-galois full subcategories of $(\C_K, \F_K)$; a proof of this general galois correspondence is also provided in \S \ref{s4}. 

It would be possible to extend the results of the current paper to more general situations, say, along the line of general Shimura varieties \cite{Milne}; in this relation, compare our results to Yalkinoglu's \cite{Yalkinoglu12} and Ha and Paugam's work \cite{Ha_Paugam}; in terms of general Shimura varieties the current work concerns the case of Siegel modular varieties (= Shimura varieties for general symplectic group) in particular.  In this paper, however, we do not pursue such full generality yet; instead we shall focus on the most basic case where the common/different points of CM fields to the case of imaginary quadratic fields \cite{Uramoto20} become transparent.

\paragraph{Related works}
The current work is motivated by Yalkinoglu's work \cite{Yalkinoglu12} (2012) in particular. In that paper, Yalkinoglu constructed an arithmetic subalgebra of the Bost-Connes system for number fields containing CM fields using Siegel modular functions on Siegel modular varieties (\S 7 \cite{Yalkinoglu12}); and if we restrict his arithmetic subalgebra (consisting of certain functions on a \emph{BCM groupoid}; cf.\ \S 6 \cite{Yalkinoglu12}) to the subspace of its objects (which is isomorphic to $DR_K$), we can in fact obtain a subalgebra of $E_K$ under the identification $E_K = \Hom_{G_K}(DR_K, K^{ab})$; in other words, up to this identification, his construction of arithmetic subalgebra can be also understood as a construction of algebraic Witt vectors over CM fields from special values of Siegel modular functions; this way, Yalkinoglu's work \cite{Yalkinoglu12} is relevant to our major goal in the current paper. 

Comparing our construction to Yalkinoglu's \cite{Yalkinoglu12}, however, highlights the above mentioned fact that our result can be regarded as a deformation of Shimura's reciprocity law. On the one hand, the algebraic Witt vectors arising from Yalkinoglu's construction, viewed as functions on $DR_K$, have their support only on the unit group $DR_K^\times = G_K^{ab} \subset DR_K$ (\S 7 \cite{Yalkinoglu12}), while those of our construction have their support on the whole $DR_K$ (cf.\ \S \ref{s3}--\S \ref{s4}). Technically speaking, it is due to this aspect that our algebraic Witt vectors can relate special values of Siegel modular functions living at different levels that are not necessarily galois conjugate, unlike Shimura's reciprocity law only relates special values of galois-conjugate functions. 

To be more specific, an algebraic Witt vector $\xi \in E_K$, viewed as a function on $DR_K$, is locally constant and thus factors through a finite quotient $DR_K \twoheadrightarrow DR_\idf$ for some conductor $\idf$; this finite monoid $DR_\idf$ has a natural action of its unit group $DR_\idf^\times$ isomorphic to the ray class group $C_\idf$, and then, has a $C_\idf$-orbit decomposition of the form $DR_\idf = \coprod_{\idd \mid \idf} C_{\idf/\idd}$; the $G_K$-equivariance of $\xi$ together with $K^{ab}$-valuedness implies that $\xi$ is a function $\coprod_{\idd} C_{\idf/\idd} \rightarrow K^{ab}$ such that, when restricted on each component $\xi|_{C_{\idf/\idd}}: C_{\idf/\idd} \rightarrow K^{ab}$, its value $\xi(\sigma)$ at $\sigma \in C_{\idf/\idd}$ is equal to $\xi(1)^\sigma$ and thus determined only by its value at the unit $1 \in C_{\idf/\idd}$. This way, the values of $\xi$ within a single component $C_{\idf/\idd}$ are correlated by galois conjugate; but $\xi$ being an algebraic Witt vector, the values of $\xi$ on the whole $DR_\idf = \coprod_\idd C_{\idf/\idd}$ (i.e.\ those on different components too) must be correlated so that the distribution of the values of $\xi$ on the whole $DR_\idf = \coprod_\idd C_{\idf/\idd}$ are \emph{arithmetically smooth} (i.e.\ infinitely differentiable with respect to arithmetic derivatives) up to fixed denominator (i.e.\ $\xi$ multiplied by some integer gets arithmetically smooth). In this sense, an algebraic Witt vector $\xi$ is a data consisting, generally, not only of galois-conjugate values but connecting (or smoothly \emph{interpolating}) non-galois conjugate values too, where ``smooth'' is in the sense of arithmetic derivative. 
Phrased in these specific terms, the algebraic Witt vectors arising from Yalkinoglu's construction consist only of galois-conjugate values (= special values of galois conjugate Siegel modular functions) by his definition, and thus, in the above general sense, do not give interpolation of special values of non-galois-conjugate modular functions; on the other hand, our construction induces such interpolations in a natural way.\footnote{That said, it should be also recognized that Yalkinoglu's work has his own goal that belongs to a natural research context on Bost-Connes systems, and for this goal, his construction was sufficient; our above argument is just to see that, reviewing Yalkinoglu's work in the above specific way, our research subject (interpolations of special values of non-galois-conjugate Siegel modular functions) emerges naturally; solidifying this subject itself, as a sequel to our previous work \cite{Uramoto20}, is also a major theme of the current work.}

Technically speaking, it seems characteristic that, in our construction of algebraic Witt vectors (i.e.\ modular vectors) by special values of Siegel modular functions (living at different levels), moduli interpretations of Siegel modular varieties (or Siegel upper-half space $\upper_g$) play an intrinsic key role, which is also in a sharp contrast to Yalkinoglu's work. In fact, as mentioned above, our modular vectors over CM fields (including imaginary quadratic fields) are constructed by choosing specific generators of the field $\modular$ of modular functions, namely, Fricke functions $f_a$ ($a \in \ratf^2/\integer^2$) and ratios $\theta^{k/l}_a$ ($a \in \ratf^{2g}/\integer^{2g}$) of theta functions (cf.\ \S \ref{s3}); and these specific modular functions have natural moduli interpretation that they represent coordinates of torsion points of (projectively embedded) elliptic curves and abelian varieties. Based on this choice of generator, we can deform modular functions by deforming the indices $a \in \ratf^{2g}/\integer^{2g}$ (cf.\ \S \ref{s3}) of the generators (i.e.\ $f_a$ and $\theta^{k/l}_a$), which correspond to descending the levels of torsion points of abelian varieties; our modular vectors are constructed by special values of thus deformed modular functions. In a sense, the reason why the values of our modular vectors $\xi: \coprod_{\idd \mid \idd} C_{\idf/\idd} \rightarrow K^{ab}$ can be correlated even among different components is that their values are controlled uniformly by this well-behaved deformation family $\theta^{k/l}_a$; in fact we can even characterize what algebraic Witt vectors are generated by our modular vectors, by geometry of abelian varieties, thanks to their geometric origin (cf.\ \S \ref{s4}).

In the case of imaginary quadratic fields, the current work also overlaps with our own previous work \cite{Uramoto20}; in particular, we now have two classes of modular vectors, i.e.\ those by Fricke functions $f_a$ and those by the ratios $\theta^{k/l}_a$ of theta functions. Although the latter cannot generate all algebraic Witt vectors over CM field of higher degree in general, we can see by our result in \S \ref{s4} that, at least in the case of imaginary quadratic fields, the latter too can generate all algebraic Witt vectors, as in the former case \cite{Uramoto20}; so, in view of the fact that the latter makes sense for any CM fields, modular vectors by the ratios of theta functions have a theoretical advantage/canonicity; in fact, our choice of ratios of theta functions for general CM case is a natural consequence of Mumford's general theory of projective embeddings of abelian varieties \cite{Mumford} and Shimura's theory of complex multiplication \cite{Shimura_abelian, Shimura_arithmeticity}. Also as Iwao Kimura has told us in the context of our work \cite{scss2021} with Yasuhiro Ishitsuka on the decidability of integral Witt vectors over imaginary quadratic fields, numerical computations of special values of Fricke functions can be made fast by representing them by theta series \cite{Johansson_Enge_Hart}. In view of this, the modular vectors by ratios of theta functions may also have a computational advantage even for the case of imaginary quadratic fields only. 
Also, in the case of CM fields, Streng \cite{Streng} has developed some computational results on Shimura's reciprocity law for Siegel modular functions, which is also relevant to the current work (\S \ref{s3}--\S \ref{s4}); in particular, Streng's result may be helpful for concrete numerical computations on our general theory of modular vectors developed below. 

\paragraph{Acknowledgements}
This paper is an extended version of our talk at \emph{Low Dimensional Topology and Number Theory XIV} on 27--30, March 2023. We are grateful to the organizer of the workshop for inviting us there; we are also grateful to Yasuhiro Ishitsuka and Iwao Kimura for several discussions on \cite{scss2021}. This work was supported by JSPS KAKENHI Grant-No.22K03248. 

\section{Preliminaries}
\label{s2}
Throughout this paper, $K$ denotes a number field, and $O_K$ the ring of integers in $K$; we denote by $P_K$ the set of non-zero prime ideals of $O_K$, while $I_K$ denotes the monoid of non-zero integral ideals of $O_K$; $N\ida$ denotes the absolute norm of $\ida \in I_K$. Let $\Ad_K$ be the adele ring of $K$, $\Ad_{K, f}$  the finite adeles, and $\widehat{O}_K$ the finite integral adeles; the Artin map is denoted $[-]: \Ad_K^{\times} \rightarrow G_K^{ab}$, where $G_K^{ab}$ is the galois group of the maximal abelian extension $K^{ab}$ over $K$; for a ring $R$ with unit, $R^\times$ denotes the invertible elements in $R$. Note that this notation is different from that in \cite{Uramoto20}, where we denoted $R^*$ to mean the invertible elements of $R$; in this paper, the superscript $*$ is reserved for the \emph{reflex fields} $E^*$ for CM fields with CM types $(E, \Phi)$ (cf.\ \S \ref{s3.2}). 

In this paper the reader is assumed to be familiar with the terminologies, notations, and main results in our previous works \cite{Uramoto16, Uramoto17, Uramoto18, Uramoto20}. In particular, $\intwitt$ denotes the $\Lambda$-ring of integral Witt vectors over $O_K$, while $E_K = K \otimes \intwitt$ denotes the $\Lambda$-ring of algebraic Witt vectors over $K$. Also, $(\C_K, \F_K)$ denotes the semi-galois category of finite etale $\Lambda$-rings over $K$ having integral models, whose fundamental monoid $\pi_1(\C_K, \F_K)$ is isomorphic to the Deligne-Ribet monoid $DR_K$ \cite{Borger_Smit11}.

\subsection{Modular vectors}
\label{s2.1}
This subsection is devoted to an overview of our previous construction of modular vectors (\S 4 \cite{Uramoto20}), rephrasing it in terms of elliptic curves and their Kummer varieties in particular. This rephrase is to highlight the geometric ingredient of modular vectors, and to compare them to the case of CM fields given in \S \ref{s3}, where elliptic curves (i.e.\ abelian varieties of dimension one) are replaced with abelian varieties of higher dimension (cf.\ \S \ref{s2.2}). 

\subsubsection{Witt vectors as functions on $DR_K$}
As outlined in \S \ref{s1}, the construction of modular vectors is based on the fact that we can identify the algebraic Witt vectors $\xi \in E_K$ with locally constant $K^{ab}$-valued $G_K$-equivariant functions on the profinite Deligne-Ribet monoid $DR_K$, which holds for arbitrary number fields $K$ (\S 3 \cite{Uramoto20}). Since we need this identification in this paper too, let us start with recalling some results on it:

\begin{defn}[Deligne-Ribet monoid $DR_K$]
The \emph{(profinite) Deligne-Ribet monoid} $DR_K$ is defined as the inverse limit of finite quotient monoids $I_K / \sim_\idf =: DR_\idf$ given for each $\idf \in I_K$, where $\sim_\idf$ is the monoid congruence of finite index on the monoid $I_K$ defined by, for $\ida, \idb \in I_K$, 
\begin{eqnarray}
 \ida \sim_\idf \idb &\Leftrightarrow& \exists t \in K_+ \cap (1 + \idf \idb^{-1}), \hspace{0.2cm} \ida \idb^{-1} = (t),
\end{eqnarray}
where $K_+$ denotes totally positive elements in $K$. 
Note that, for $\idf \mid \idf'$, we have a natural surjective monoid homomorphism $DR_{\idf'} \twoheadrightarrow DR_\idf$, with which $DR_\idf$'s constitute an inverse system of finite monoids; with this inverse system $\{DR_\idf\}_{\idf \in I_K}$ of \emph{finite Deligne-Ribet monoids $DR_\idf$ for conductors $\idf \in I_K$}, the profinite Deligne-Ribet monoid $DR_K$ is given by the inverse limit: 
\begin{eqnarray}
 DR_K &:=& \lim_{\leftarrow \idf \in I_K} DR_\idf
\end{eqnarray}
\end{defn}

By definition we have a canonical monoid homomorphism $I_K \rightarrow DR_K$, which is in fact injective; also, it is known \cite{Deligne_Ribet} that the unit group $DR_\idf$ of each $DR_\idf$ is isomorphic to the strict ray class group $C_\idf$ for the conductor $\idf$, namely $DR_\idf^\times \simeq C_\idf$; and hence, that $DR_K^\times \simeq \lim_\idf C_\idf \simeq G_K^{ab}$, where the last isomorphism is the one given by class field theory. Therefore, the absolute galois group $G_K$ naturally acts on $DR_K$ (say, from the right) via the maximal abelian quotient $G_K \twoheadrightarrow G_K^{ab} \simeq DR_K^\times \subseteq DR_K$, with which $DR_K$ forms a profinite right $G_K$-set. Thanks to this fact, it makes sense to consider locally constant $K^{ab}$-valued \emph{$G_K$-equivariant} functions on $DR_K$, the set of which we shall denote by $\Hom_{G_K}(DR_K, K^{ab})$. 

This set $\Hom_{G_K}(DR_K, K^{ab})$ has a natural $K$-algebra structure with respect to the point-wise operations, as well as the following operation $\psi_\idp: \Hom_{G_K}(DR_K, K^{ab}) \rightarrow \Hom_{G_K}(DR_K, K^{ab})$ for each $\idp \in P_K$, which is given by multiplication-by-$\idp$, regarding each $\idp \in P_K$ as an element in $DR_K$ via the above-mentioned injection $I_K \hookrightarrow DR_K$:
\begin{eqnarray}
 \Hom_{G_K}(DR_K, K^{ab}) &\xrightarrow{\psi_\idp}& \Hom_{G_K}(DR_K, K^{ab}) \\
 \xi(*) &\longmapsto& \xi(\idp \cdot *)
\end{eqnarray}
With these operators $\psi_\idp$, the $K$-algebra $\Hom_{G_K}(DR_K, K^{ab})$ forms a $\Lambda$-ring over $K$, which in fact is isomorphic to the $\Lambda$-ring $E_K$ of algebraic Witt vectors in a canonical way:

\begin{thm}[Corollary 3.1.6, \cite{Uramoto20}]
\label{identification}
We have a canonical isomorphism of $\Lambda$-rings over $K$:
\begin{eqnarray}
 E_K &\simeq& \Hom_{G_K}(DR_K, K^{ab}). 
\end{eqnarray}
\end{thm}

Technically, the proof of this theorem is based on the two facts (cf.\ \S 3.2 \cite{Uramoto20}) that (i) each finite set $\Xi \subseteq E_K$ of algebraic Witt vectors generates a \emph{galois object} $X_\Xi$ of the semi-galois category $\C_K$, where $X_\Xi := K[\psi_\ida \xi \mid \xi \in \Xi, \ida \in I_K]$ is the $\Lambda$-ring generated over $K$ by the orbit of $\Xi$ under the action of $\psi_\idp$'s; and (ii) every galois object of $\C_K$ is isomorphic to $X_\Xi$ for some finite $\Xi \subseteq E_K$. (See loc.\ cit.\ for more detail.)

In particular, the isomorphism from $\Hom_{G_K}(DR_K, K^{ab})$ to $E_K$ is defined by restricting locally constant $K^{ab}$-valued $G_K$-equivariant functions $\xi: DR_K \rightarrow K^{ab}$ onto the submonoid $I_K \hookrightarrow DR_K$, inducing functions $\xi : I_K \rightarrow K^{ab}$, or in other words, vectors $\xi \in (K^{ab})^{I_K}$; the above theorem says that the vector $\xi \in (K^{ab})^{I_K}$ induced from $\xi \in \Hom_{G_K}(DR_K, K^{ab})$ always defines an algebraic Witt vector, namely $\xi \in E_K \subset (K^{ab})^{I_K}$ more than $\xi \in (K^{ab})^{I_K}$; and conversely, every algebraic Witt vector $\xi \in E_K$ arises uniquely in this way.
In the sense of this correspondence, we can now identify algebraic Witt vectors $\xi \in E_K$ with locally constant $K^{ab}$-valued $G_K$-equivariant functions on the profinite monoid $DR_K$; in particular, to construct algebraic Witt vectors $\xi \in E_K$ is equivalent to constructing such functions $\xi \in \Hom_{G_K}(DR_K, K^{ab})$. 

\subsubsection{Modular vectors}
When $K$ is an imaginary quadratic field, our previous work \cite{Uramoto20} constructed algebraic Witt vectors $\xi \in E_K$ from special values of the family $\{f_a\}$ of Fricke functions $f_a: \upper \rightarrow \comp$, ($a \in \ratf^2/\integer^2$), which are called \emph{modular vectors}. In the rest of this subsection, we shall review this construction in terms of elliptic curves and their Kummer varieties, which is suitable for extension to abelian varieties of higher dimension (\S \ref{s3}). 

To this end, an adelic description of $DR_K$ essentially by Yalkinoglu \cite{Yalkinoglu13} plays an auxiliary role: 

\begin{prop}[Proposition 8.2, Yalkinoglu \cite{Yalkinoglu13}]
We have a natural isomorphism of profinite monoids:
\begin{eqnarray}
\label{adelic}
 DR_K &\simeq& \widehat{O}_K \times_{\widehat{O}_K^\times} \bigl(\Ad_{K}^\times / \overline{K^\times \cdot K_\infty^{\times, \circ}} \bigr)
\end{eqnarray}
\end{prop}
Here $K^{\times, \circ}_\infty$ denotes the connected component of the archimedian part $K_\infty^\times$ of the idele group $\Ad^\times_K$ containing the unit $1$; the overline on $K^\times \cdot K^{\times, \circ}_\infty$ denotes the closure in $\Ad_K^\times$. Also the multiplicative group $\widehat{O}_K^\times$ acts on the product monoid $\widehat{O}_K \times (\Ad_K^\times/ \overline{K^\times \cdot K^{\times, \circ}_\infty})$ by $u \cdot (\rho, s) := (\rho u, u^{-1}s)$, for which we take the quotient in (\ref{adelic}). In particular, when $K$ is imaginary quadratic, we have the following simpler form:
\begin{eqnarray}
 DR_K &\simeq& \widehat{O}_K \times_{\widehat{O}_K^\times} \bigl(\Ad_{K,f}^\times / K^\times \bigr)
\end{eqnarray}
Generally, for CM fields $K$, we can represent elements of $DR_K$ as $[\rho, s]$ with $\rho \in \widehat{O}_K$ and $s \in \Ad_{K,f}^\times$. 

The subject here is to review geometrically our previous construction in \cite{Uramoto20} of locally constant $K^{ab}$-valued $G_K$-equivariant functions on $DR_K$ from \emph{Fricke modular functions} $f_a: \upper \rightarrow \comp$ indexed by $a \in \ratf^2/\integer^2$. Technically the Fricke functions $f_a: \upper \rightarrow \comp$ can be described explicitly but in a case-by-case manner according as $K = \ratf(i), \ratf(\sqrt{-3})$ or others (cf.\ \S 6.1 \cite{Shimura}); this is related to the fact that the automorphism groups of elliptic curves with complex multiplication in $K$ differ for the respective cases (cf.\ \S 4.5 \cite{Shimura}). It is for this reason that our study of modular vectors was also done in a case-by-case manner (cf.\ e.g.\ \S 4.3 \cite{Uramoto20}). 

Geometrically, we have a more uniform description of modular vectors via \emph{Kummer varieties}, which is equal to the original one. The \emph{Kummer variety} $V_E$ of an elliptic curve $E$ is the quotient of $E$ by its automorphism group, $V_E = E/\Auto(E)$, which is again an algebraic curve; and the quotient map $h: E \twoheadrightarrow V_E$ is called the \emph{Weber function}; the Kummer varieties $V_E$ are in fact isomorphic to the projective space $\proj^1(\comp)$ in such a way that we have a natural geometric interpretation of the Fricke functions $f_a$: That is, for each non-zero $a = (a_1, a_2) \in \ratf^2/\integer^2$ and $\tau \in \upper$, the special value $f_a(\tau)$ is (the $x$-coordinate of) the image of the torsion point $a_1 \tau + a_2 \in \comp/(\integer \tau + \integer)$ under the map $\comp/(\integer \tau + \integer) \rightarrow \proj^1(\comp)$ given by the following composition:
\begin{equation}
\xymatrix{\comp/(\integer \tau + \integer) \ar[r]^{\hspace{0.6cm}[\wp, \wp', 1]} & E_\tau \ar@{->>}[r]^h & V_E \ar[r]^{\hspace{-1.2cm}\simeq} & \proj^1(\comp) = \comp \cup \{\infty\}};
\end{equation}
also, for $a = 0 \in \ratf^2/\integer^2$, we put $f_0 (\tau) = j(\tau)$. 

With this geometric interpretation of the Fricke functions $f_a$, we now construct locally constant $K^{ab}$-valued $G_K$-equivariant functions $\widehat{f}_a$ on $DR_K$, which provide our target modular vectors:

\begin{defn}[modular vectors $\widehat{f}_a$]
For each $a = (a_1, a_2) \in \ratf^2/\integer^2$, define $\widehat{f}_a: DR_K \rightarrow \comp$ by the following rule: Given each $[\rho, s] \in \widehat{O}_K \times_{\widehat{O}_K^\times} (\Ad_{K,f}^\times/K ^\times) \simeq DR_K$, the value $\widehat{f}_a (\rho, s)$ is the image of the torsion point $a_1 \tau_K + a_2 \in K/O_K$ under the map $K/O_K \rightarrow \proj^1(\comp)$ given by the composition:
\begin{equation*}
\xymatrix{K/O_K \ar[r]^{\rho} \ar@{-}[d]_\simeq & K/O_K \ar[r]^s \ar@{-}[d]_\simeq & K/sO_K \ar@{-}[d]_\simeq \ar@{->}[r]^{\subseteq} & \comp/sO_K \ar[r]^\simeq & E_s \ar[r]^{\hspace{-1.3cm}h} & \proj^1(\comp) = \comp \cup \{ \infty\} \\ 
\Ad_{K,f}/\widehat{O}_K \ar[r]_\rho & \Ad_{K, f}/\widehat{O}_K \ar[r]_s & \Ad_{K, f}/ s \widehat{O}_K & & && }
\end{equation*}
where $O_K = \integer \tau_K + \integer$ with $\tau_K \in \upper$ and $sO_K := s \widehat{O}_K \cap K$ for each $s \in \Ad^\times_{K,f}$; and the isomorphism $\comp/sO_K \simeq E_s$ is given by the Weierstrass $\wp$-function $\wp(u, \Lambda_s)$ for the lattice $\Lambda_s = sO_K \subseteq \comp$; also, if the value $\widehat{f}_a(\rho, s)$ goes to $\infty$ by this map, we renormalize it as $j(\Lambda_s)$. 
\end{defn}

\begin{prop}
These functions $\widehat{f}_a: DR_K \rightarrow \comp$ are equal to those in \cite{Uramoto20}, hence define locally constant $K^{ab}$-valued $G_K$-equivariant functions on $DR_K$, or equivalently, modular vectors $\widehat{f}_a \in E_K$. 
\end{prop}
\begin{proof}
The only slight difference from the original construction is that, in the above construction, we did not re-scale the lattice $sO_K$ into the form $\integer \tau + \integer$ with $\tau \in \upper$; therefore, the lattice $sO_K$ is of the form $sO_K = \integer w_1 + \integer w_2$ with $\tau_s = w_1/w_2 \in \upper$ in general. But this difference in scaling does not affect the resulting values because the complex torus $\comp/sO_K$ is isomorphic to the rescaled one $\comp/ (\integer \tau_s + \integer)$; and the Weber functions on elliptic curves are compatible with their isomorphisms (cf.\ \S 4.5 \cite{Shimura}). This also means that the above construction is independent of the choice of $s \in \Ad^\times_{K,f}$ that represents the class $[s] \in \Ad^\times_{K, f}/K^\times$, hence $\widehat{f}_a: DR_K \rightarrow \comp$ is well-defined.
\end{proof}

\subsection{Basic facts on CM fields}
\label{s2.2}
Now we prepare some basic results necessary to extend the above construction of modular vectors to the case of CM fields. In this process, as outlined in \S \ref{s1} too, we replace elliptic curves with abelian varieties of higher dimension; accordingly, we need to recall projective embedding of complex tori of higher dimension by theta functions \cite{Mumford}, which replaces projective embedding of one-dimensional complex tori $\comp/(\integer w_1 + \integer w_2)$ by the Weierstrass $\wp$-function $\wp$ (and $\wp'$). Starting from a few basic facts on CM fields, this subsection gathers some results related to projective embedding of complex tori with complex multiplication in CM fields; for more details on these topics, the reader is referred to e.g.\ \cite{Birkenhake_Lange, Milne, Shimura_abelian, Shimura_arithmeticity}. 

\subsubsection{CM fields and Riemann forms}
\label{s2.2.1}
When $K$ is an imaginary quadratic field, the complex tori $\comp/sO_K \simeq K_\infty / sO_K$ are polarizable for any $s \in \Ad^\times_{K,f}$, which is not the case for complex tori of higher dimension in general. However, when $K$ is a CM field (say, of degree $[K: \ratf] = 2g$), the complex tori $\comp^g/sO_K \simeq K_\infty/sO_K$ are always polarizable, hence, can be embedded into projective space as abelian varieties. 

To be precise, let $K$ be a CM field of degree $[K: \ratf] = 2g$; also, fix a \emph{CM type} $\Phi = \{\phi_1, \cdots, \phi_g\}$. That is, $\phi_i: K \hookrightarrow \comp$ are embeddings that are neither equal nor conjugate to each other; therefore, denoting by $\bar{\phi}_i: K \hookrightarrow \comp$ the complex conjugate of $\phi_i$, the set $\{\phi_1, \cdots, \phi_g, \bar{\phi}_1, \cdots, \bar{\phi}_g\}$ is precisely that of all embeddings $K \hookrightarrow \comp$. We abusively use the same symbol $\Phi: K \rightarrow \comp^g$ to denote also the map $\Phi: K \ni a \mapsto (\phi_1(a), \cdots, \phi_g(a)) \in \comp^g$; by this map $\Phi$, the $O_K$-modules $sO_K := s\widehat{O}_K \cap K \subseteq K$ for $s \in \Ad^\times_{K, f}$ are embedded into $\comp^g$ as lattices $\Lambda_s := \Phi(sO_K) \subseteq \comp^g$, which then induce the complex tori $\comp^g/\Lambda_s$. 
The first fact necessary in this paper is that the complex tori $\comp^g/\Lambda_s$ are \emph{polarizable} in the following sense, for any $s \in \Ad^\times_{K,f}$ of the CM field $K$:

\begin{defn}[polarization on complex tori]
For a lattice $\Lambda \subseteq \comp^g$, the complex torus $\comp^g / \Lambda$ is said to be \emph{polarizable} if there exists an alternating $\realf$-bilinear form $\psi: \comp^g \times \comp^g \rightarrow \realf$ such that:
\begin{enumerate}
\item $\psi(u, v) \in \integer$ for any $u, v \in \Lambda$;
\item $\psi(J(u), J(v)) = \psi(u, v)$ for any $u, v \in \comp^g$;
\item $\psi(u, J(u)) > 0$ for any non-zero $u \in \comp^g$;
\end{enumerate}
where $J: \comp^g \rightarrow \comp^g$ is the standard complex structure. Such a form $\psi$ is called a (non-degenerate) \emph{Riemann form} on the complex torus $\comp^g/\Lambda$. 
\end{defn}

The fact that the complex tori $\comp^g/\Lambda_s$ are polarizable is based on a few structural properties of the CM field $K$ (e.g.\ Theorem 4, \S 6.2 \cite{Shimura_abelian}): Being a CM field, $K$ is by definition a totally imaginary quadratic extension of its maximal totally real subfield $K_0$; take a totally imaginary element $e \in K$ such that $-e^2$ is a totally positive element of $K_0$ and $\imm(\phi_i (e)) > 0$ for all $\phi_i \in \Phi$, where $\imm(z)$ denotes the imaginary part of $z \in \comp$. With this $e$, define $\psi_e: \comp^g \times \comp^g \rightarrow \realf$ by:
\begin{eqnarray}
 \psi_e(u, v) &:=& \sum_{i=1}^g \phi_i(e) \cdot (u_i \bar{v}_i - \bar{u}_i v_i)
\end{eqnarray}
for $u=(u_1, \cdots, u_g), v=(v_1, \cdots, v_g) \in \comp^g$. In particular, for each $a, b \in K$, we have:
\begin{eqnarray}
 \psi_e(\Phi(a), \Phi(b)) &=& \tr_{K/\ratf} (e \cdot a \cdot b^\iota),
\end{eqnarray}
where $\tr_{K/\ratf}$ denotes the trace over $\ratf$; and $\iota \in Gal(K/K_0)$ denotes the non-trivial element of order 2. In particular, this means that, the $O_K$-module $sO_K$ for any $s \in \Ad^\times_{K,f}$ being a fractional ideal of $K$, we have $n_s\cdot  \psi_e(sO_K, sO_K) \subseteq \integer$ for some integer $n_s$; and then, this $\psi:= n_s\cdot \psi_e$ gives a Riemann form on the complex torus $\comp^g/ \Lambda_s$. 

\begin{rem}[symplectic basis and type]
Given any Riemann form $\psi$ on a complex torus $\comp^g/\Lambda$ in general, the lattice $\Lambda$ has a \emph{symplectic basis}, which is a basis of the lattice $\Lambda$ over $\integer$ by which the Riemann form $\psi: \comp^g \times \comp^g \rightarrow \realf$ can be expressed by the matrix of the form:
\begin{eqnarray}
 \psi(u, v) &=& u \cdot \begin{pmatrix} 0 & \delta \\ -\delta & 0 \end{pmatrix} \cdot v^t
\end{eqnarray}
where $u, v \in \comp^g$ in the left-hand side are re-represented in the right-hand side as real (row) vectors $u, v \in \realf^{2g}$ with respect to the symplectic basis of $\Lambda$; also $\delta \in M_g (\integer)$ is a matrix of the form:
\begin{eqnarray}
 \delta &=& \begin{pmatrix} d_1 & 0 & \cdots & 0 \\ 0 & d_2 & & 0 \\ \vdots && \ddots & \vdots \\ 0 & 0 & \cdots & d_g   \end{pmatrix}
\end{eqnarray}
with $0 < d_i$ and $d_i \mid d_{i+1}$; in the following we denote $\delta = \diag[d_1, \cdots, d_g]$. This $\delta$ is independent of the choice of symplectic basis of $\Lambda$ for the Riemann form $\psi$; thus it makes sense to say that the Riemann form $\psi$ on the complex torus $\comp^g/\Lambda$ is \emph{of type $\delta$}. 
\end{rem}

\subsubsection{Projective embedding}
\label{s2.2.2}
By the above fact that the complex torus $\comp^g / \Lambda_s$ is polarizable for any $s \in \Ad^\times_{K,f}$ of the CM field $K$ with a fixed CM type $\Phi$, it follows that the complex torus $\comp^g/\Lambda_s$ can be isomorphically embedded as an abelian variety in a suitable projective space $\proj^d(\comp)$. Such an embedding can be concretely provided by \emph{theta functions} \cite{Mumford}, which depends on a choice of Riemann form on the complex torus $\comp^g/\Lambda_s$ and its type.

To see such embeddings and theta functions explicitly, take a lattice $\Lambda \subseteq \comp^g$ such that we have a Riemann form $\psi$ on the complex torus $\comp^g/\Lambda$ of type $\delta = \diag[d_1, \cdots, d_g]$; then the lattice $\Lambda$ has a symplectic basis $(\tau_1, \cdots, \tau_g, \delta_1, \cdots, \delta_g)=: (\tau, \delta) \in M_{g \times 2g}(\comp)$ so that $\Lambda =  \integer^g \cdot \tau +  \integer^g \cdot \delta$, where $\integer^g$ consists of row vectors and $\tau_i \in \comp^g, \delta_i \in \integer^g$ are column vectors. Here the $g \times g$ matrix $\tau \in M_g(\comp)$ is in fact a member of the \emph{Siegel upper-half space $\upper_g$} defined by:
\begin{eqnarray}
 \upper_g &:=& \biggl\{ \tau \in M_g(\comp) \mid \tau^t = \tau, \imm(\tau) > 0 \biggr\}
\end{eqnarray}
where $\imm(\tau)>0$ means that the imaginary part $\imm(\tau) \in M_g(\realf)$ of $\tau$ is positive definite. We note conversely that, given $\tau \in \upper_g$, one obtains a lattice $\Lambda_\tau := \integer^g \tau + \integer^g \delta$ such that the complex torus $\comp^g / \Lambda_\tau$ has a Riemann form $\psi_\tau$ of type $\delta$, given by $\psi_\tau (u, v) := \imm(u \cdot (\imm(\tau))^{-1} \cdot \bar{v}^t)$; then $(\tau, \delta)$ gives a symplectic basis of $\Lambda_\tau$ for this Riemann form $\psi_\tau$. Therefore, it follows that the complex tori $\comp^g/\Lambda$ with Riemann forms $\psi$ of type $\delta$ are precisely those of the form $\comp^g/(\integer^g \tau + \integer^g \delta)$ for some $\tau \in \upper_g$ (cf.\ e.g.\ \S 8.1 \cite{Birkenhake_Lange}).

Now we are ready for projective embeddings of the complex tori $\comp^g/(\integer^g \tau + \integer^g \delta)$ with Riemann forms $\psi$ of type $\delta$ by the following \emph{theta functions}:

\begin{defn}[theta function $\theta^k$]
\label{theta}
For each $k \in \delta^{-1} \integer^g/\integer^g$, define $\theta^k: \comp^g \times \upper_g \rightarrow \comp$ as follows:
\begin{eqnarray}
 \theta^k (u; \tau) &:=& \sum_{w \in k + \integer^g} \exp(\frac{w^t \cdot \tau \cdot w}{2} + u \cdot w),
\end{eqnarray}
where $\exp(t) := e^{2 \pi i t}$; also $u \in \comp^g$ are row vectors, while $k \in \delta^{-1}\integer^g/\integer^g$ and $w \in k + \integer^g$ are supposed to be column vectors. 
\end{defn}

The theta functions $\theta^k$ define projective embedding of complex tori of type $\delta = \diag[d_1, \cdots, d_g]$ when $3 \leq d_1$ in the following manner:

\begin{defn}[projective embedding]
\label{projective embedding}
Put $d := |\delta^{-1}\integer^g/\integer^g| -1 = d_1 d_2 \cdots d_g$ - 1; then for each $\tau \in \upper_g$, we can define $\Theta_\tau: \comp^g/ (\integer^g \tau + \integer^g \delta) \rightarrow \proj^d (\comp)$ by:
\begin{eqnarray}
 \comp^g / (\integer^g \tau + \integer^g \delta) &\xrightarrow{\Theta_\tau}& \proj^d(\comp) \\
 u &\longmapsto& [\theta^k(u; \tau)]_{k \in \delta^{-1}\integer^g / \integer^g}
\end{eqnarray}
where $[\theta^k]_{k \in \delta^{-1} \integer^g / \integer^g}$ denotes the homogeneous coordinate on the projective space $\proj^d(\comp)$; by our assumption on $\delta$, this map $\Theta_\tau: \comp^g / (\integer^g \tau + \integer^g \delta) \rightarrow \proj^d(\comp)$ defines a projective embedding, whose image $\Theta_\tau (\comp^g/(\integer^g \tau + \integer^g\delta)) \subseteq \proj^d(\comp)$ is an abelian variety and will be denoted $A_\tau \subseteq \proj^d(\comp)$, thus $\Theta_\tau: \comp^g/(\integer^g \tau + \integer^g \delta) \xrightarrow{\sim} A_\tau$. 
\end{defn}

\subsubsection{Complex multiplication}
\label{s2.2.3}
In the case when $K$ is an imaginary quadratic field, the special value $f(\tau)$ of a modular function on $\upper$ (rational over $\ratf^{ab}$) at $\tau \in K \cap \upper$ is abelian over $K$; Shimura's reciprocity law then shows that the galois conjugate $(f(\tau))^{[s]^{-1}}$ of this special value $f(\tau)$ for $s \in \Ad_{K,f}^\times$ is equal to the special value $f^{q_\tau(s)} (\tau)$ of a suitably defined galois conjugate modular function $f^{q_\tau(s)}$ (cf.\ Theorem 6.18 \cite{Shimura}). In terms of elliptic curves, this reciprocity law is proved using the following \emph{main theorem of complex multiplication}:

\begin{thm}[main theorem of complex multiplication, Theorem 5.4 \cite{Shimura}]
\label{cm for elliptic curve}
Let $\ida$ be a fractional ideal of an imaginary quadratic field $K \subseteq \comp$; and $\Theta: \comp/\ida \rightarrow E_\ida$ an isomorphism of the complex torus $\comp/\ida$ to an elliptic curve $E_\ida$. Then, for any automorphism $\sigma: \comp \rightarrow \comp$ over $K$ with $\sigma|_{K^{ab}} = [s]^{-1}$ for some $s \in \Ad^\times_{K,f}$, there exists a unique isomorphism $\Theta': \comp/s\ida \rightarrow E^\sigma$ such that the following diagram commutates:
\begin{equation}
\xymatrix{ K/\ida \ar[d]_{s} \ar[r] & \comp/\ida  \ar[r]^\Theta & E_\ida \ar[d]^{\sigma} \\
K/s\ida \ar[r] &\comp/s\ida \ar[r]_{\Theta'} & E_\ida^{\sigma}}
\end{equation}
\end{thm}

In the case when $K$ is a CM field of higher degree in general, a similar theorem still holds, where elliptic curves are replaced with abelian varieties of higher dimension; but in order to state it in the case of degree greater than $2$, we need to consider not only $K$ itself but also its \emph{reflex field} $K^*$ for a CM type of $K$:

\begin{defn}[reflex field]
Let $K$ be a CM field of degree $[K: \ratf] = 2g$ and $\Phi = \{\phi_1, \cdots, \phi_g\}$ a CM type of $K$. Then the reflex field $K^*$ of the CM type $(K, \Phi)$ is defined as follows:
\begin{eqnarray}
 K^* &:=& \ratf \biggl(\sum_{i=1}^g \phi_i (e) \mid e \in K \biggr). 
\end{eqnarray}
\end{defn}

\begin{rem}[reflex norm]
\label{reflex norm}
For the CM type $(K, \Phi)$, the reflex field $K^*$ is again a CM field; also, we have the \emph{reflex norm} for $(K, \Phi)$, which is a multiplicative monoid morphism $\eta: \Ad_{K^*} \rightarrow \Ad_K$; this morphism restricts to $\eta: \Ad_{K^*, f}^\times \rightarrow \Ad_{K,f}^\times$ and $\eta: \widehat{O}_{K^*} \rightarrow \widehat{O}_K$; for more details the reader is referred to e.g.\ \S 18.5 \cite{Shimura_abelian} or \S 9.4 \cite{Shimura_arithmeticity}. 
\end{rem}

\begin{thm}[main theorem of complex multiplication, Theorem 18.6 \cite{Shimura_abelian}]
\label{cm for abelian variety}
Let $\ida$ be a fractional ideal of a CM field $K$; and $\Theta: \comp^g/\Phi(\ida) \rightarrow A_\ida$ an isomorphism of the complex torus $\comp^g/\Phi(\ida)$ to an abelian variety $A_\ida$ (\S \ref{s2.2.2}). Then, for any automorphism $\sigma: \comp \rightarrow \comp$ over $K^*$ with $\sigma|_{(K^*)^{ab}} = [s]^{-1}$ for some $s \in \Ad_{K^*, f}^\times$, there exists a unique isomorphism $\Theta': \comp^g/\Phi(\eta(s) \ida) \rightarrow A_\ida^\sigma$ such that the following diagram commutes:
\begin{equation}
\xymatrix{ E/\ida \ar[r]^{\hspace{-0.4cm}\Phi}  \ar[d]_{\eta(s)} & \comp^g/\Phi(\ida) \ar[r]^{\hspace{0.3cm}\Theta} & A_\ida \ar[d]^\sigma \\
E/\eta(s)\ida \ar[r]_{\hspace{-0.4cm}\Phi} & \comp^g/\Phi(\eta(s)\ida) \ar[r]_{\hspace{0.4cm} \Theta'} & A_\ida^\sigma }
\end{equation}
\end{thm}

\section{Modular vectors over CM fields}
\label{s3}
The subject of this section is to construct \emph{modular vectors} in the case when $K$ is the reflex field of a CM type; to be compatible with notations in \S \ref{s1}, let $(E, \Phi)$ be a CM type of degree $[E: \ratf] = 2g$, and put $K = E^*$; by the identification discussed in \S \ref{s2.2.1}, it suffices to construct locally constant $K^{ab}$-valued $G_K$-equivariant functions on $DR_K$. Thus, we divide this section into two sub-sections for (i) a construction of suitable functions on $DR_K$ from modular functions on the Siegel upper-half space $\upper_g$ (\S \ref{s3.1}), and (ii) a proof that they are indeed locally constant, $K^{ab}$-valued, and $G_K$-equivariant (\S \ref{s3.2}). 

\subsection{Construction of modular vectors}
\label{s3.1}
Throughout this section, we shall fix a CM field $E$ of degree $[E: \ratf] = 2g$ together with a CM type $\Phi = \{\phi_1, \cdots, \phi_g\}$ of $E$; let $K = E^*$ be the reflex field of $(E, \Phi)$. Also we fix a Riemann form on the complex torus $\comp^g/\Phi(O_E)$, $\psi: \comp^g \times \comp^g \rightarrow \realf$, so that the ring $O_E$ of integers has a symplectic basis $(\tau_1, \cdots, \tau_g, \delta_1, \cdots, \delta_g)$ for $\psi$ with which $\psi$ is of type $\delta = \diag[d_1, \cdots, d_g]$ and $3 \leq d_i \mid d_{i+1}$. Such a Riemann form $\psi: \comp^g \times \comp^g \rightarrow \realf$ always exists; the condition on its type $\delta$ is necessary to guarantee that the theta functions $\theta^k$ induce  projective embeddings of complex tori $\comp^g/ (\integer^g \tau + \integer^g \delta)$ of type $\delta$ for each $\tau \in \upper_g$ (cf.\ \S \ref{s2.2.2}); for a technical reason (cf.\ Proposition 8.10.2 \cite{Birkenhake_Lange}), we further suppose that $4 \leq d_1$ and $2 \mid d_1$ or $3 \mid d_1$. 

In particular, by construction, the ring $O_E$ of integers in $E$ is mapped to a lattice $\Phi(O_E) \subseteq \comp^g$ by $\Phi: E \rightarrow \comp^g$, which can be written by a matrix $\tau_E \in \upper_g$ as in the following form:
\begin{eqnarray}
 \Phi(O_E) &=& \integer^g \tau_E + \integer^g \delta
\end{eqnarray}
Therefore, as discussed in \S \ref{s2.2.3}, the complex torus $\comp^g/\Phi(O_E) = \comp^g/(\integer^g \tau_E + \integer^g \delta)$ is isomorphic to an abelian variety, which we shall denote $A_E \subseteq \proj^d(\comp)$, through the following projective embedding $\Theta_E:=\Theta_{\tau_E}: \comp^g/\Phi(O_E) \hookrightarrow \proj^d(\comp)$ given by (cf.\ Definition \ref{projective embedding}):
\begin{eqnarray} 
\comp^g/\Phi(O_E) &\xrightarrow{\Theta_{E}}& \proj^d(\comp)\\
u &\longmapsto& [\theta^k (u; \tau_E)]_{k \in \delta^{-1}\integer^g/\integer^g},
\end{eqnarray}
where $d = d_1 d_2 \cdots d_g - 1$. More generally, for any $s \in \Ad_{K, f}^\times$, the $O_E$-module $\eta(s)O_E := \eta(s)\widehat{O}_E \cap E$ is mapped to the lattice $\Phi(\eta(s)O_E) \subseteq \comp^g$, where $\eta: \Ad_{K, f}^\times = \Ad_{E^*, f}^\times \rightarrow \Ad_{E, f}^\times$ denotes the reflex norm (cf.\ \S \ref{s2.2.3}); in the following we denote this lattice $\Phi(\eta(s)O_E)$ by $\Lambda_s$ for each $s \in \Ad_{K,f}^\times$, in particular, $\Lambda_E := \Lambda_1 = \Phi(O_E)$. 

The main theorem of complex multiplication (Theorem \ref{cm for abelian variety}) shows that, for each automorphism $\sigma: \comp \rightarrow \comp$ over $K=E^*$ with $\sigma|_{K^{ab}} = [s]^{-1}$, we have a unique isomorphism $\Theta': \comp/\Lambda_s \rightarrow A_E^\sigma$ such that:
\begin{equation}
\xymatrix{ E/O_E \ar[r]^{\Phi}  \ar[d]_{\eta(s)} & \comp^g/\Lambda_E \ar[r]^{\hspace{0.3cm}\Theta_E} & A_E \ar[d]^\sigma \\
E/\eta(s)O_E \ar[r]_{\Phi} & \comp^g/\Lambda_s \ar[r]_{\hspace{0.4cm} \Theta'} & A_E^\sigma }
\end{equation} 
We will use these $\Theta'$ for our construction of modular vectors; but note that, while the isomorphism $\Theta_E: \comp^g/\Lambda_E \rightarrow A_E$ is given explicitly by the theta functions $\theta^k(*; \tau_E)$ as above, the isomorphism $\Theta': \comp^g/\Lambda_s \rightarrow A_E^\sigma$ is not yet quite explicit for $s \in \Ad^\times_{K, f}$. Therefore, our task here is first to construct modular vectors over $K$ using these abstract $\Theta'$; and then to describe the resulting modular vectors more explicitly. 

Concerning this first task, notice that, while $\Theta': \comp^g/\Lambda_s \rightarrow A_E^\sigma$ itself depends on $\sigma$ rather than $\sigma|_{K^{ab}}= [s]^{-1}$ (hence, is not actually determined only by $s$), the composition $\Theta' \circ \Phi: E/\eta(s)O_E \hookrightarrow \comp^g/\Lambda_s \rightarrow A_E^\sigma \subseteq \proj^d(\comp)$ does not because the image of $E/\eta(s)O_E$ are the torsion points of $A_E^\sigma$ and the action of $\sigma$ on the torsion points is determined by $[s]$. Therefore, it makes sense to denote this composition as $\Theta_s: E/\eta(s)O_E \rightarrow A_E^\sigma \subseteq \proj^d(\comp)$ for each $s \in \Ad_{K,f}^\times$. With this in mind, we abusively denote as $A_E^\sigma = A_s$, and also $\Theta_s: \comp^g/\Lambda_s \rightarrow A_s$. 

Now we construct modular vectors using these (still abstract) isomorphisms $\Theta_s: \comp^g/\Lambda_s \rightarrow A_s$. To this end, we note that each abelian variety $A_s$ is a priori a subvariety of the ambient projective space $\proj^d(\comp)$; thus the rational functions on $\proj^d(\comp)$ uniformly restrict to each subvariety $A_s$. More explicitly, in terms of the homogeneous coordinates $[\theta^k]_{k \in \delta^{-1}\integer^g/\integer^g}$ of $\proj^d(\comp)$, the rational functions on $\proj^d(\comp)$ are of the form $f(\theta^k) = g(\theta^k)/h(\theta^k)$, where $g, h$ are homogeneous polynomials of $\theta^k$'s of the same degree; in particular, we mean by \emph{rational functions (over $K$)} those $f(\theta^k) = g(\theta^k)/h(\theta^k)$ whose coefficients are in $K$. Also, for each $a = (a_1, a_2) \in \ratf^{2g}/\integer^{2g}$, we say that a rational function $f(\theta^k)$ is \emph{defined at $a$} if it is defined on the torsion point of $A_E$ corresponding to $a_1 \tau_E + a_2 \delta \in E/O_E$ under the map $\Theta_E \circ \Phi : E/O_E \hookrightarrow \comp^g/\Lambda_E \simeq A_E$; also we say that $f(\theta^k)$ is \emph{defined for $a$} if $f(\theta^k)$ is defined at $a \eta(\rho)$ for any $\rho \in \widehat{O}_K$, where and in the following we identify $a = (a_1, a_2) \in \ratf^{2g}/\integer^{2g}$ with $a_1 \tau_E + a_2 \delta \in E/O_E$. 

\begin{defn}[modular vector]
\label{defn of modular vector}
Let $a \in N^{-1}\integer^{2g}/\integer^{2g}$ and $f$ a rational function defined for $a$. Then the \emph{modular vector} $\widehat{f}_a: DR_K \rightarrow \comp$ is defined by the following rule: Given each $[\rho, s] \in DR_K$, the value $\widehat{f}_a(\rho, s)$ is the image of the torsion point $a_1 \tau_E + a_2 \delta \in E/O_E$ (at the upper-left corner in (\ref{diagram for modular vector})) under the map $E/O_E \rightarrow \comp$ given by the following composition:
\begin{equation}
\label{diagram for modular vector}
\xymatrix{E/O_E \ar[r]^{\eta(\rho)} \ar@{-}[d]_\simeq & E/O_E \ar[r]^{\eta(s)} \ar@{-}[d]_\simeq & E/\eta(s)O_E \ar@{-}[d]_\simeq \ar@{->}[r]^{\Phi} & \comp^g/\Lambda_s \ar[r]^{\Theta_s} & A_s \ar[r]^{f|_{A_s}} & \comp \\ 
\Ad_{E,f}/\widehat{O}_E \ar[r]_{\eta(\rho)} & \Ad_{E, f}/\widehat{O}_E \ar[r]_{\eta(s)} & \Ad_{E, f}/ \eta(s) \widehat{O}_E & & && }
\end{equation}
\end{defn}

\begin{rem}[rational functions over $\ratf$]
In the definition of modular vectors, we used rational functions $f$ defined over $K$; but it suffices for our purpose in \S \ref{s4} to consider only rational functions defined over $\ratf$. In the latter case, nevertheless, the set of modular vectors itself does not form a $K$-algebra; in that case, we need to consider the $K$-algebra generated (over $K$) by modular vectors.
\end{rem}

\begin{lem}
\label{well-definedness}
The above modular vectors $\widehat{f}_a$ are well-defined as $\comp$-valued functions on $DR_K$.
\end{lem}
\begin{proof}
We need to check two points for the well-definedness of $\widehat{f}_a$: The first is that the value $\widehat{f}_a(\rho, s)$ is indeed finite; the second is that it depends only on the class $[\rho, s] \in DR_K$ rather than $(\rho, s)$. To see these, consider the following commutative diagram:
\begin{equation}
\xymatrix{E/O_E \ar[r]^{\eta(\rho)} & E/O_E \ar[d]_{\eta(s)} \ar[r]^\Phi & \comp^g/\Lambda_E \ar[r]^{\Theta_E} & A_E \ar[d]^{[s]^{-1}} \\
& E/\eta(s)O_E \ar[r]_\Phi & \comp^g/\Lambda_s \ar[r]_{\Theta_s} & A_s
}
\end{equation}
Since the right-most vertical isomorphism $[s]^{-1}: A_E \rightarrow A_s$ (on the torsion points) depends only on the class $[s]$, it first follows that the value $\widehat{f}_a(\rho, s)$ depends on $[s]$ rather than $s$ itself; also, the torsion point $a_1 \tau_E + a_2 \delta \in N^{-1}O_E/O_E$ is mapped again to $N^{-1}O_E/ O_E$ by $\eta(\rho): E/O_E \rightarrow E/O_E$; therefore, the assumption that $f$ is now defined for $a$ implies that the value $\widehat{f}_a(\rho, s)$ is finite. 
Finally the equality $\widehat{f}_a(\rho u, u^{-1} s) = \widehat{f}_a (\rho, s)$ for any $u \in \widehat{O}_K^\times$ follows from the fact that the composition $\eta(u^{-1}s) \circ \eta(\rho u): E/O_E \rightarrow E/\eta(u^{-1}s)O_E = E/\eta(s)O_E$ is equal to $\eta(s) \circ \eta(\rho): E/O_E \rightarrow E/\eta(s)O_E$ since the reflex norm $\eta$ is multiplicative. This completes the proof.
\end{proof}

\subsection{Basic properties}
\label{s3.2}
We prove here that the modular vectors $\widehat{f}_a: DR_K \rightarrow \comp$ are indeed locally constant $K^{ab}$-valued and $G_K$-equivariant, hence define algebraic Witt vectors $\widehat{f}_a \in E_K$ over $K$. Moreover, as mentioned above too, the construction of modular vectors depends on the abstract isomorphisms $\Theta_s: \comp^g/\Lambda_s \rightarrow A_s$; thus, we shall also give a more explicit description of modular vectors in terms of the theta functions $\theta^k: \comp^g \times \upper_g \rightarrow \comp$ of \S \ref{s2.2.2} at a concrete $\tau_s \in \upper_g$. 

We start with proving the fact that the modular vectors define algebraic Witt vectors:

\begin{prop}
\label{basic property}
Take a rational function $f$ and $a \in N^{-1}\integer^{2g}/\integer^{2g}$ as in Definition \ref{defn of modular vector}; then the modular vector $\widehat{f}_a: DR_K \rightarrow \comp$ is locally constant $K^{ab}$-valued and $G_K$-equivariant.
\end{prop}
\begin{proof}
We first see that $\widehat{f}_a: DR_K \rightarrow \comp$ is $K^{ab}$-valued. To this end, note the following commutative diagram:
\begin{equation}
\xymatrix{E/O_E \ar[r]^{\eta(\rho)} & E/O_E \ar[d]_{\eta(s)} \ar[r]^\Phi & \comp^g/\Lambda_E \ar[r]^{\Theta_E} & A_E \ar[d]^\sigma \ar[r]^{f|_{A_E}} & \comp \ar[d]^\sigma \\
& E/\eta(s)O_E \ar[r]_\Phi & \comp^g/\Lambda_s \ar[r]_{\Theta_s} & A_s \ar[r]_{f|_{A_s}} & \comp
}
\end{equation}
where $\sigma|_{K^{ab}} = [s]^{-1}$. By definition of $\widehat{f}_a$, this diagram shows that we have:
\begin{eqnarray}
\widehat{f}_a(\rho, s) &=& f(\theta^k(a_1' \tau_E + a_2' \delta; \tau_E))^\sigma
\end{eqnarray}
where $a' = (a_1', a_2') \in N^{-1}\integer^{2g}/\integer^{2g}$ is defined as $a' = a \cdot \eta(\rho)$. Since $f(\theta^k)=g(\theta^k)/h(\theta^k)$ is now a rational function over $K$ with homogeneous polynomials $g, h$ of the same degree, it follows that this value $\widehat{f}_a(\rho, s)$ can be expressed as an algebraic combination of ratios of theta values in the following form for some $k, l \in \delta^{-1}\integer^g/\integer^g$:
\begin{eqnarray}
\biggl( \frac{\theta^k(a_1' \tau_E + a_2' \delta; \tau_E)}{\theta^l (a_1' \tau_E + a_2' \delta; \tau_E)} \biggr)^\sigma
\end{eqnarray}
In fact, these ratios are abelian over $K$ since the following meromorphic function $\theta^{k/l}_a$ on $ \upper_g$ defines a modular function on $\upper_g$ for all $a \in \ratf^{2g}/\integer^{2g}$ and $k, l \in \delta^{-1}\integer^g/\integer^g$ (cf.\ Theorem 27.13, \cite{Shimura_abelian}):
\begin{eqnarray}
\label{theta ratio}
 \theta^{k/l}_a (\tau) &:=& \frac{\theta^k(a_1 \tau + a_2 \delta; \tau)}{\theta^l (a_1 \tau + a_2 \delta; \tau)} 
\end{eqnarray}
Therefore, the action of $\sigma$ on these values at $\tau_E$ is in fact equal to $[s]^{-1}$; and we eventually have:
\begin{eqnarray}
\label{galois conjugate expression of modular vector}
\widehat{f}_a(\rho, s) &=& f(\theta^k(a_1' \tau_E + a_2' \delta; \tau_E))^{[s]^{-1}} \in K^{ab}. 
\end{eqnarray}
From this expression, we also easily deduce that $\widehat{f}_a: DR_K \rightarrow K^{ab}$ is $G_K$-equivariant. (Recall that the action of $G_K$ on $DR_K$ factors through $G_K \twoheadrightarrow G_K^{ab}$, and the action of $[t] \in G_K^{ab}$ on $[\rho, s]$ is equal to $[\rho, t^{-1}s]$; see e.g.\ Lemma 4.2.3 \cite{Uramoto20}.) 

Finally, we see that $\widehat{f}_a:DR_K \rightarrow K^{ab}$ is locally constant; in fact, to be more specific, $\widehat{f}_a$ factors through the finite quotient $DR_K \twoheadrightarrow DR_{NO_K}$ for the conductor $NO_K$ for $a \in N^{-1} \integer^{2g}/\integer^{2g}$. To see this, note first that the action of $\eta(\rho) \in \widehat{O}_E$ (by multiplication) on $N^{-1}O_E/O_E$ is determined by the class $[\rho] \in \widehat{O}_K/N \widehat{O}_K$ (hence also by the classes in $\widehat{O}_K/N' \widehat{O}_K$ for any $N \mid N'$); this follows from the fact that the reflex norm $\eta: \widehat{O}_K = \widehat{O}_{E^*} \rightarrow \widehat{O}_E$ maps $N\widehat{O}_K$ into $N \widehat{O}_E$ (cf.\ \S 18.5, \cite{Shimura_abelian}). Moreover, the values $\widehat{f}_a(\rho, s)$ belong at largest to the ray class field $K_{N'O_K}$ of the conductor $N'O_K$ for some $N'$ such that $N \mid N'$ since these values are given by special values of $f(\theta^k)$ at $N$-torsion points of $A_E$, hence abelian over $K$ in particular; thus the action of $[s]^{-1}$ on these values factors through the quotient $G_K^{ab} \twoheadrightarrow C_{N'O_K}$. Combining with the isomorphism $DR_{N'O_K} \simeq (O_K/N'O_K) \times_{(O_K/N'O_K)^\times} C_{N'O_K}$ (cf.\ Proposition 7.2 \cite{Yalkinoglu13}), we deduce that $\widehat{f}_a$ indeed factors through a finite quotient $DR_K \twoheadrightarrow DR_{N'O_K}$. This completes the proof. 
\end{proof}

As seen above, the proof of the fact that the modular vectors $\widehat{f}_a$ are locally constant $K^{ab}$-valued $G_K$-equivariant was possible even with the abstract isomorphisms $\Theta_s: \comp^g/\Lambda_s \rightarrow A_s$. But our next task is to give a more explicit description of the values $\widehat{f}_a(\rho, s)$ by special values of ratios of theta functions $\theta^k(*; \tau_s)$ at an \emph{explicit} $\tau_s \in \upper_g$ rather than the galois conjugates $\theta^k(*; \tau_E)^{[s]^{-1}}$ as given in the right hand side of (\ref{galois conjugate expression of modular vector}); technically, this task is done by finding $\tau_s \in \upper_g$ such that the lattice $\Lambda_s = \Phi(\eta(s)O_E)$ is isomorphic to $\integer^g \tau_s + \integer^g \delta$. 

To this end, note first that, for each $s \in \Ad_{K,f}^\times$, the $\Ad_{\ratf, f}$-linear map $\Ad_{E, f} \ni t \mapsto \eta(s) \cdot t \in \Ad_{E, f}$ can be expressed as a $2g \times 2g$ matrix in $GL_{2g}(\Ad_{\ratf, f})$ with respect to the symplectic basis of $O_E$ for the Riemann form $\psi$, which we abusively denote also as $\eta(s) \in GL_{2g}(\Ad_{\ratf,f})$; this adelic matrix $\eta(s)$ has the following standard matrix decomposition, with which we will construct the target $\tau_s \in \upper_g$: 

\begin{lem}[cf.\ strong approximation theorem]
\label{decomposition}
Put $\tilde{\eta}(s) \in GL_{2g}(\Ad_{\ratf,f})$ as follows:
\begin{eqnarray}
\label{adjoint of eta}
 \tilde{\eta}(s)  &:=& \begin{pmatrix} 1 & 0 \\ 0 & \delta \end{pmatrix}^{-1} \cdot \eta(s) \cdot  \begin{pmatrix} 1 & 0 \\ 0 & \delta \end{pmatrix}
\end{eqnarray}
Then we have $\tilde{\eta}(s) = \tilde{u}_s \cdot \tilde{\alpha}_s$ with some $\tilde{u}_s \in GL_{2g}(\Ad_{\ratf, f})$ and $\tilde{\alpha}_s \in GSp^+_{2g}(\ratf)$ such that we have the equality $(\widehat{\integer}^g + \widehat{\integer}^g\delta) \cdot \tilde{u}_s = \widehat{\integer}^g + \widehat{\integer}^g \delta$.
\end{lem}
\begin{proof}
First, by the basic property of the reflex norm (cf.\ Remark \ref{reflex norm}), the $\Ad_{\ratf, f}$-linear map $\eta(s)$ in fact preserves the Riemann form $\psi=\psi_{\Ad_{\ratf, f}}: \Ad_{E, f} \times \Ad_{E, f} \rightarrow \Ad_{\ratf, f}$ up to a fixed scalar over $\Ad_{\ratf, f}$. That is, for any $u, v \in \Ad_{E, f}$:
\begin{eqnarray}
 \psi(u \cdot \eta(s), v \cdot \eta(s)) &=& \psi(u, v \cdot \eta(s)\eta(s)^\iota)\\
 &=& \psi(u, v\cdot N_{K/\ratf}(s))\\
 &=& N_{K/\ratf}(s) \cdot \psi(u, v) 
\end{eqnarray}
By definition of symplectic basis for the Riemann form $\psi$, this means that, in terms of the matrix $\eta(s) \in GL_{2g}(\Ad_{\ratf, f})$: 
\begin{eqnarray}
 \eta(s) \cdot \begin{pmatrix} 0 & \delta \\ - \delta & 0 \end{pmatrix} \cdot \eta(s)^t &=& N_{K/\ratf}(s) \cdot \begin{pmatrix} 0 & \delta \\ - \delta & 0 \end{pmatrix}
\end{eqnarray}
Note also that we have:
\begin{eqnarray}
\begin{pmatrix} 0 & \delta \\ -\delta & 0 \end{pmatrix} &=& \begin{pmatrix} 1 & 0 \\ 0 & \delta \end{pmatrix} \cdot \begin{pmatrix} 0 & 1 \\ -1 & 0 \end{pmatrix} \cdot \begin{pmatrix} 1 & 0 \\ 0 & \delta \end{pmatrix} 
\end{eqnarray}
From this, we easily deduce $\tilde{\eta}(s) \in GSp_{2g}(\Ad_{\ratf,f})$. The decomposition in the claim is a consequence of the strong approximation theorem for $GSp_{2g}(\Ad_{\ratf, f})$ (cf.\ Lemma 8.3 and Lemma 9.5 \cite{Shimura_arithmeticity}). 
\end{proof}

For each $s \in \Ad_{K,f}^\times$, let $u_s$ and $\alpha_s$ be the similar transformations of $\tilde{u}_s$ and $\tilde{\alpha}_s$ as the above $\eta(s)$ of $\tilde{\eta}(s)$, whence $\eta(s) = u_s \cdot \alpha_s$ and $(\widehat{\integer}^g + \widehat{\integer}^g) u_s = \widehat{\integer}^g + \widehat{\integer}^g$; also, partition $\alpha_s$ into the following form:
\begin{eqnarray}
 \alpha_s &=& \begin{pmatrix} a_s & b_s \\ c_s & d_s \end{pmatrix}.
\end{eqnarray}
We then define $\tau_s:= \alpha_s (\tau_E)$ as $\tau_s := (a_s \tau_E + b_s\delta) \cdot (c_s \tau_E + d_s \delta)^{-1} \delta$, which is indeed a member of $\upper_g$ by $\tau_E \in \upper_g$ and $\tilde{\alpha}_s \in GSp^+_{2g}(\ratf)$; we see that this $\tau_s \in \upper_g$ is our target one. 

\begin{prop}[see also Theorem 2.4, Streng \cite{Streng}]
With notations as in Definition \ref{defn of modular vector}, we have the following equation:
\begin{eqnarray}
\label{target}
 \widehat{f}_a (\rho, s) &=& f(\theta^k (a''_1 \tau_s + a_2'' \delta; \tau_s)). 
\end{eqnarray}
where $\tau_s = \alpha_s(\tau_E)$ and $a'' = (a_1'', a_2'') \in \ratf^{2g}/\integer^{2g}$ is defined as the image of $a \in \ratf^{2g}/\integer^{2g}$ under the following map:
\begin{equation}
\xymatrix{\ratf^{2g}/\integer^{2g} \ar[r]^{\eta(\rho)} &\ratf^{2g}/\integer^{2g} \ar[r]^{u_s} & \ratf^{2g}/\integer^{2g}}.
\end{equation}
That is, $a'' = a \cdot \eta(\rho) u_s \in N^{-1}\integer^{2g}/\integer^{2g}$. 
\end{prop}
\begin{proof}
Again recall the following diagram defining the value of $\widehat{f}_a(\rho, s)$:
\begin{equation}
\xymatrix{E/O_E \ar[r]^{\eta(\rho)} & E/O_E \ar[d]_{\eta(s)} \ar[r]^\Phi & \comp^g/\Lambda_E \ar[r]^{\Theta_E} & A_E \ar[d]^\sigma \ar[r]^{f|_{A_E}} & \comp \ar[d]^\sigma \\
& E/\eta(s)O_E \ar[r]_\Phi & \comp^g/\Lambda_s \ar[r]_{\Theta_s} & A_s \ar[r]_{f|_{A_s}} & \comp
}
\end{equation}
Here the lattice $\Lambda_s = \Phi(\eta(s) O_E)$ is equal to $(\integer^g, \integer^g) \cdot u_s \alpha_s \cdot (\tau_E, \delta)^t = (\integer^g, \integer^g) \cdot \alpha_s \cdot (\tau_E, \delta)^t$ by our construction (Lemma \ref{decomposition}), which is further isomorphic to $(\integer^g, \integer^g) \cdot (\tau_s, \delta)^t = \integer^g \tau_s + \integer^g \delta$ by the right action of $\chi_s := (c_s \tau_E + d_s\delta)^{-1}\delta$. With this in mind, we consider yet another diagram:
\begin{equation}
\xymatrix{E/O_E \ar[r]^{\eta(\rho)} & E/O_E \ar[d]_{\eta(s)} \ar[r]^\Phi & \comp^g/\Lambda_E \ar[rr]^{\Theta_E} & & A_E \ar[d]^{[s]^{-1}} \ar[r]^{f|_{A_E}} & \comp \ar[d]^\sigma \\
& E/\eta(s)O_E \ar[r]_\Phi & \comp^g/\Lambda_s \ar[r]^{\hspace{-0.6cm}\chi_s}_{\hspace{-0.8cm}\simeq} & \comp^g/(\integer^g \tau_s + \integer^g \delta) \ar[r]_{\hspace{0.8cm}\simeq}^{\hspace{1cm}\Theta_{\tau_s}} & A_{\tau_s}  \ar[r]_{f|_{A_{\tau_s}}} & \comp
}
\end{equation}
By the defining property of $\Theta_s: \comp^g/\Lambda_s \rightarrow A_s$, in order to prove (\ref{target}), it suffices to prove that this diagram commutes. 

We see that this commutativity follows from the fact that the ratios $\theta^{k/l}_a$ (\ref{theta ratio}) of homogeneous coordinates $[\theta^k]$ at torsion points of $A_E \subseteq \proj^d(\comp)$ define modular functions on $\upper_g$ (cf.\ Theorem 27.13 \cite{Shimura_abelian}), together with Shimura's reciprocity law (Theorem 9.6 \cite{Shimura_arithmeticity}). 
That is, to be more specific, first note that since $f(\theta^k) = g(\theta^k)/h(\theta^k)$ is a rational function of homogeneous degree $0$, it can be in fact composed of the ratios $\theta^k/\theta^l$ of theta functions $\theta^k$; thus, $f(\theta^k(a_1 \tau + a_2 \delta; \tau))$ also defines a modular function on $\upper_g$ for each fixed $a \in \ratf^{2g}/\integer^{2g}$, which we may denote as $f(\theta^{k/l}_a(\tau))$. Moreover, since $f$ is defined over $K$ and $\sigma$ fixes $K$, we have $f(\theta_a^{k/l}(\tau))^\sigma = f((\theta_a^{k/l}(\tau))^{[s]^{-1}})$. Then applying Shimura's reciprocity law for these modular functions $\theta^{k/l}_a$ and $[s] \in G_K^{ab}$, we deduce that the middle rectangle commutes; in fact see Lemma 8.6 and Theorem 8.10 \cite{Shimura_arithmeticity}\footnote{Note that we use the lattice $\integer^g + \integer^g \delta$ rather than $\integer^g + \integer^g$; accordingly, we use matrices transformed by $\begin{pmatrix} 1 & 0\\ 0 & \delta \end{pmatrix}$ when we apply these theorems.} for a careful comparison of the definition of galois conjugate $(\theta^{k/l}_a)^{\eta(s)}$ of the modular functions $\theta^{k/l}_a$ with our composition of the above maps; in particular we have $(\theta^{k/l}_a)^{\eta(s)} (\tau_E) = \theta^{k/l}_{a u_s}(\alpha_s(\tau_E)) = \theta^{k/l}_{a u_s} (\tau_s)$. This completes the proof. 
\end{proof}

\section{Galois correspondence}
\label{s4}
Unlike the case of imaginary quadratic fields, the modular vectors constructed above do not generate all algebraic Witt vectors over a CM field $K$ of higher degree; therefore, it becomes reasonable to ask what algebraic Witt vectors are generated by our modular vectors. Denoting by $M_K$ the set of modular vectors, we can see by definition that $M_K$ is in fact a subalgebra of $E_K$ over $K$; moreover, it can be easily seen that $M_K$ is a $\Lambda$-subalgebra of $E_K$ (i.e.\ closed under the action of $\psi_\idp$'s), which follows from the same argument as Lemma 4.3.5 \cite{Uramoto20}. Therefore, the \emph{general galois correspondence} mentioned in Remark 3.1.8 \cite{Uramoto20} implies that we have a dual characterization of modular vectors in terms of profinite quotient of $DR_K$. 

This section is devoted to this problem and divided into two subsections. In the first subsection (\S \ref{s4.1}), we give a proof of the general galois correspondence mentioned in Remark 3.1.8 \cite{Uramoto20} without proof, which claims a canonical bijective correspondence between (i) $\Lambda$-subalgebras of $E_K$, (ii) profinite quotients of $DR_K$, and (iii) semi-galois full subcategories of $\C_K$. This general correspondence, however, is not quite satisfactory in itself in that it just tells us only an existence of such a quotient but not its explicit structure. To remedy this issue, the second subsection (\S \ref{s4.2}) is devoted to an explicit characterization of modular vectors, where we give an adelic description of the profinite quotient of $DR_K$ that corresponds to the $\Lambda$-subalgebra of modular vectors $M_K \subseteq E_K$ under our general galois correspondence. 

\subsection{General galois correspondence}
\label{s4.1}
In this subsection, we allow $K$ to be an arbitrary number field; our goal here is to prove the following general galois correspondence mentioned in Remark 3.1.8 \cite{Uramoto20}:

\begin{prop}[galois correspondence]
\label{galois correspondence}
There exists a canonical bijective correspondence between the following three classes of objects:
\begin{enumerate}
\item $\Lambda$-subalgebras of $E_K$;
\item profinite quotients of $DR_K$; and
\item semi-galois full subcategories of $\C_K$. 
\end{enumerate}
\end{prop}

Note that the correspondence between the second and the third is just a direct consequence of the general duality theorem between profinite monoids and semi-galois categories \cite{Uramoto16,Uramoto17} together with the fact that the fundamental monoid $\pi_1(\C_K, \F_K)$ of $\C_K$ is isomorphic to $DR_K$ \cite{Borger_Smit08, Borger_Smit11}; therefore, in order to prove this proposition, it suffices to show a bijective correspondence between the first and the second only. For this proof, however, the concept of \emph{galois objects} in the semi-galois category $\C_K$ plays an auxiliary key role. 

We start with the correspondence from the first to the second; and to this end, we need to recall a few basic concepts (with notations slightly different from those in \cite{Uramoto20}). 

\begin{defn}[monoid congruence on $I_K$]
Let $\Xi \subseteq E_K$ be a finite subset of $E_K$. Then, define the monoid congruence $\ida \sim_\Xi \idb$ for $\ida, \idb \in I_K$ as follows:
\begin{eqnarray}
 \ida \sim_\Xi \idb &\Leftrightarrow& \forall \xi \in \Xi, \hspace{0.2cm} \psi_\ida \xi = \psi_\idb \xi. 
\end{eqnarray}
Also, denote by $D_\Xi$ the finite quotient monoid $D_\Xi := I_K / \sim_\Xi$. 
\end{defn}

Concerning this monoid congruence, note that, for two finite subsets $\Xi, \Xi' \subseteq E_K$ with $\Xi \subseteq \Xi'$, we have a canonical surjective monoid homomorphism $r^{\Xi'}_\Xi: D_{\Xi'} \twoheadrightarrow D_\Xi$; therefore, with respect to the inclusion relation $\Xi \subseteq \Xi'$ of finite subsets of $E_K$, we have an inverse system $\{D_\Xi, r^{\Xi'}_\Xi\}$ of finite monoids $D_{\Xi}$. The fact that $E_K$ is the union of the galois objects $X_\Xi$ of $\C_K$ (\S 3.1 \cite{Uramoto20}) implies in particular that $DR_K$ is isomorphic to the inverse limit of this system $\{D_\Xi, r^{\Xi'}_\Xi\}$ when $\Xi \subseteq E_K$ runs over all finite subsets of $E_K$; thus by restricting the range where $\Xi \subseteq E_K$ are allowed to run, we naturally obtain a quotient of $DR_K$: 

\begin{defn}[dual to $\Lambda$-subalgebra of $E_K$]
Let $E \subseteq E_K$ be a $\Lambda$-subalgebra of $E_K$; then define $D_E$ as the following profinite monoid:
\begin{eqnarray}
 D_E &:=& \lim_{\leftarrow \Xi \subseteq E} D_\Xi. 
\end{eqnarray}
\end{defn}

Conversely, given a profinite quotient monoid $DR_K \twoheadrightarrow D$, we construct a $\Lambda$-subalgebra $E_D \subseteq E_K$ so that this construction $D \mapsto E_D$ gives the inverse of the above construction $E \mapsto D_E$. To this end, recall (\S \ref{s2.1}) that we can identify $E_K$ with the $\Lambda$-ring $\Hom_{G_K}(DR_K, K^{ab})$ of locally constant $K^{ab}$-valued $G_K$-equivariant functions of $DR_K$; under this identification, we can naturally describe the target correspondence $D \mapsto E_D$ as follows:

\begin{defn}[dual to quotient of $DR_K$]
Let $DR_K \twoheadrightarrow D$ be a profinite quotient monoid of $DR_K$; then we define a $\Lambda$-subalgebra $E_D \subseteq E_K$ as follows:
\begin{eqnarray}
 E_D &:=& \biggl\{ \xi \in \Hom_{G_K}(DR_K, K^{ab}) \mid \textrm{$\xi$ factors through $DR_K \twoheadrightarrow D$} \biggr\}.
\end{eqnarray}
\end{defn}

\begin{lem}
For any $\Lambda$-subalgebra $E \subseteq E_K$ we have the identity $E = E_{D_E}$. 
\end{lem}
\begin{proof}
Given $\xi \in E$, the finite monoid $D_{\xi} := D_{\{\xi\}}$ is by definition a quotient of $D_E$. Since $\xi$, when viewed as a function on $DR_K$, factors through $D_\xi$, it follows that $\xi$ also factors through $D_E$, hence $\xi \in E_{D_E}$, concluding $E \subseteq E_{D_E}$. 

Conversely, given $\xi \in E_{D_E}$ that is viewed as a function on $DR_K$, the definition of $E_{D_E}$ implies that $\xi: DR_K \rightarrow K^{ab}$ factors through $DR_K \twoheadrightarrow D_E$; therefore, $D_\xi$ is a finite quotient of $D_E$. Then, by definition of $D_E$, this also implies that there exists a finite subset $\Xi \subseteq E$ such that $D_\xi$ is a finite quotient of $D_\Xi$. Since $\Hom_{G_K}(D_\Xi, K^{ab})$ is equal to the galois object $X_\Xi = K[I_K \Xi]$ (\S \ref{s2.1}), we obtain $\xi \in X_\Xi$. But since $E$ is now a $\Lambda$-subalgebra thus closed under $\psi_\idp$'s, we also have $X_\Xi \subseteq E$, hence $\xi \in X_\Xi \subseteq E$, concluding $E_{D_E} \subseteq E$. This completes the proof. 
\end{proof}

\begin{lem}
For any profinite quotient monoid $DR_K \twoheadrightarrow D$, we have a canonical isomorphism $D \simeq D_{E_D}$. 
\end{lem}
\begin{proof}
First assume that two ideals $\ida, \idb \in I_K \hookrightarrow DR_K$ go into the same element in $D$ under the quotient $DR_K \twoheadrightarrow D$; we show that $\ida, \idb$ go to the same element in $D_{E_D}$ as well under the quotient $DR_K \twoheadrightarrow D_{E_D}$. Now the assumption implies that the operators $\psi_\ida, \psi_\idb$ define the same action on $E_D = \Hom_{G_K}(D, K^{ab})$ by definition (cf.\ \S \ref{s2.1}); therefore, so they do on $X_\Xi \subseteq E_D$ for any finite subset $\Xi \subseteq E_D$; in particular, we have $\psi_\ida \xi = \psi_\idb \xi$ for any $\xi \in \Xi \subseteq X_\Xi$, which means that $\ida \sim_\Xi \idb$. Now since $\Xi$ is any finite subset of $E_D$, this eventually deduces that $\ida, \idb$ are mapped to the same element in $D_{E_D}$. 

Conversely, assume that $\ida, \idb$ are mapped onto two different elements in $D$, whence we have a finite quotient $D \twoheadrightarrow H$ so that $\ida, \idb$ are still different in $H$ as well; then $\psi_\ida$ and $\psi_\idb$ define different endomorphisms on the galois object $X_H:= \Hom_{G_K}(H, K^{ab})$ in $\C_K$, thus we have $\psi_\ida \xi \neq \psi_\idb \xi$ for some $\xi \in X_H$. Since $X_H$ is now a $\Lambda$-subalgebra of $E_D$, this means that $\ida$ and $\idb$ are not congruent with respect to $\sim_{\{\xi\}}$ for the singleton $\{\xi\} \subseteq E_D$; that is, $\ida \neq \idb$ in the finite monoid $D_{\{\xi\}}$. Again, by definition of $D_{E_D} = \lim_{\Xi \subseteq E_D} D_\Xi$, this finite monoid $D_{\{\xi\}}$ is a quotient of $D_{E_D}$. Since $\ida \neq \idb \in D_\xi$, it follows that $\ida \neq \idb \in DR_{E_D}$ too. 

The above arguments show that the two quotients $DR_K \twoheadrightarrow D$ and $DR_K \twoheadrightarrow D_{E_D}$ have the same kernel on the discrete submonoid $I_K$; but since $I_K$ is dense in $DR_K$, it follows that they have the same kernel on $DR_K$ as well. Thus they must be isomorphic.
\end{proof}

Basically, we use the above galois correspondence in the following form:

\begin{cor}[dual characterization]
Let $E \subseteq E_K$ be a $\Lambda$-subalgebra such that $E$ corresponds to a profinite quotient $DR_K \twoheadrightarrow D$ under the above galois correspondence. Then $\xi \in E_K$ is a member of $E$ if and only if $D_\xi$ is a quotient of $D$. 
\end{cor}

\subsection{Characterization of modular vectors}
\label{s4.2}
Since the modular vectors constitute a $\Lambda$-subalgebra $M_K$ of $E_K$, the general galois correspondence (Proposition \ref{galois correspondence}) guarantees that we have a profinite quotient $DR_K \twoheadrightarrow D_{M_K}$ so that an algebraic Witt vector $\xi \in E_K$ is modular if and only if $D_\xi$ is a finite quotient of $D_{M_K}$; that said, nevertheless, the general galois correspondence gives us only a somewhat tautological description of $D_{M_K}$, i.e.\ $D_{M_K} = \lim_{\Xi \subseteq M_K} D_\Xi$. In view of this issue, this subsection gives in a step-by-step manner a purely adelic description of this profinite monoid $D_{M_K}$, which is intrinsic for $K$ and independent of modular vectors $\Xi \subseteq M_K$. 

As the first step, we define finite quotients $DR_K \twoheadrightarrow D_N$ for $N \geq 1$ using projective embeddings of complex tori so that $D_N$'s constitute an inverse system of finite monoids whose inverse limit is isomorphic to our target profinite monoid $D_{M_K}$ (Lemma \ref{characterization of sim N}); this description of $D_{M_K}$ is not yet intrinsic for $K$ but we utilize it to study the structure of $D_{M_K}$ via geometry of abelian varieties. Based on this, as the second step, we describe the structure of these finite monoids $D_N$ in a purely adelic term, which provides us a $K$-intrinsic description of $D_{M_K} \simeq \lim_N D_N$ and concludes our dual ($K$-intrinsic) characterization of modular vectors (cf.\ Theorem \ref{main result}).

\begin{defn}[congruence $\sim_N$]
For $[\rho, s], [\lambda, t] \in DR_K$, we define $[\rho, s] \sim_N [\lambda, t]$ if and only if the following two parallel arrows $E/O_E \rightarrow \proj^d(\comp)$ coincide on the subset $N^{-1}O_E/O_E \subseteq E/O_E$:
\begin{equation}
\label{defining diagram of congruence N}
\xymatrix{ 
& & E/\eta(s)O_E \ar[r]^{\Phi} & \comp^g/\Lambda_s \ar[r]^{\Theta_s} & A_s  \ar[rd]^{\subseteq}& \\
E/O_E \ar@<-0.1cm>[r]_{\eta(\lambda)} \ar@<0.1cm>[r]^{\eta(\rho)} & E/O_E \ar[ru]^{\eta(s)} \ar[rd]_{\eta(t)} & & & & \proj^d(\comp) \\
 & & E/\eta(t)O_E \ar[r]_\Phi & \comp^g/\Lambda_t \ar[r]_{\Theta_t} & A_t \ar[ru]_{\subseteq} & \\
}
\end{equation}
\end{defn}

\begin{lem}
The congruence $\sim_N$ is a monoid congruence on $DR_K$ of finite order. 
\end{lem}
\begin{proof}
First to see that $\sim_N$ is a monoid congruence, it suffices to prove that $[\rho, s] \sim_N [\lambda, t]$ implies $[\rho \nu, s r] \sim_N [\lambda \nu, t r]$ for any $[\nu, r] \in DR_K$. But this follows from the following larger commutative diagram, where we used Shimura's reciprocity law (cf.\ \S \ref{s2.2.3}) for the fractional ideals $\ida = \eta(s)O_E$ and $\ida = \eta(t)O_E$ in particular; note also that $\eta(\nu): E/O_E \rightarrow E/O_E$ maps the subset $N^{-1}O_E/O_E$ to itself:
\begin{equation}
\xymatrix@C=15pt{ 
& & & E/\eta(sr)O_E \ar[r]^{\Phi} & \comp^g/\Lambda_{sr} \ar[r]^{\Theta_{sr}} & A_{sr}  \ar[rrdd]^{\subseteq}& &  \\
& & & E/\eta(s)O_E \ar[r]^{\Phi} \ar[u]_{\eta(r)} & \comp^g/\Lambda_s \ar[r]^{\Theta_s} & A_s  \ar[u]^{[r]^{-1}} \ar[rd]^{\subseteq}& & \\
E/O_E \ar[r]^{\eta(\nu)}& E/O_E \ar@<-0.1cm>[r]_{\eta(\lambda)} \ar@<0.1cm>[r]^{\eta(\rho)} & E/O_E \ar[ru]_{\eta(s)} \ar[rd]^{\eta(t)}  \ar[ruu]^{\eta(sr)} \ar[rdd]_{\eta(tr)} & & & & \proj^d(\comp) \ar[r]^{[r]^{-1}} & \proj^d(\comp) \\
& & & E/\eta(t)O_E \ar[r]_\Phi \ar[d]^{\eta(r)} & \comp^g/\Lambda_t \ar[r]_{\Theta_t} & A_t \ar[d]_{[r]^{-1}} \ar[ru]_{\subseteq} & & \\
& & & E/\eta(tr)O_E \ar[r]_\Phi & \comp^g/\Lambda_{tr} \ar[r]_{\Theta_{tr}} & A_{tr} \ar[rruu]_{\subseteq} & & 
}
\end{equation}
Also, the fact that the monoid congruence $\sim_N$ is of finite index follows from the same argument in the last part of the proof of Proposition \ref{basic property}. This completes the proof. 
\end{proof}

For $N \mid N'$, it is obvious from the definition that $[\rho, s] \sim_{N'} [\lambda, t]$ implies $[\rho, s] \sim_N [\lambda, t]$; therefore, writing $D_N := DR_K / \sim_N$, we have a surjective homomorphism $r^{N'}_N: D_{N'} \twoheadrightarrow D_N$ of finite monoids for $N \mid N'$. This way we obtain an inverse system $\{D_N, r^{N'}_N\}$ of finite monoids $D_N$, whose inverse limit is in fact isomorphic to $D_{M_K}$. 

\begin{lem}
\label{characterization of sim N}
We have a canonical isomorphism of profinite monoids, $D_{M_K} \simeq \lim_{N \geq 1} D_N$. 
\end{lem}
\begin{proof}
Let $X_N:= \{\widehat{f}_a \mid a \in N^{-1} O_E/ O_E, \textrm{$f$ is rational over $K$ and defined for $a$}\} \subseteq M_K$; since $M_K$ is the union of $X_N$'s, it suffices to prove that $X_N$ is the galois object of $\C_K$ dual to $D_N$. First, by the fact that all $\widehat{f}_a$ in $X_N$ as functions on $DR_K$ factor through $DR_{NO_K}$, it follows that $X_N$ is finite etale over $K$; also, by the same argument of Lemma 4.3.5 \cite{Uramoto20}, it can be readily seen that $X_N$'s are closed under $\psi_\idp$'s. From these we deduce that $X_N = X_{\Xi(N)}$ for some finite set $\Xi(N) \subseteq X_N \subseteq M_K$, hence $X_N$ is a galois object of $\C_K$. 

The rest task is to prove that this galois object $X_N$ is dual to $D_N$; since we already know that $X_N$ is galois, this is equivalent to proving that  for any $[\rho, s], [\lambda, t] \in DR_K$, we have $[\rho, s] \sim_N [\lambda, t]$  if and only if the equality $\widehat{f}_a(\rho, s) = \widehat{f}_a(\lambda, t)$ holds for all $a \in N^{-1}O_E/O_E$ and rational functions $f$ over $K$ defined for $a$. 

The only-if part is almost trivial: notice from the definition that $[\rho, s] \sim_N [\lambda, t]$ means that each $N$-torsion point $a_1 \tau_E + a_2 \delta \in N^{-1}O_E / O_E$ is mapped to the same point in $\proj^d(\comp)$ under the two maps $E/O_E \rightarrow \proj^d(\comp)$ given in (\ref{defining diagram of congruence N}); therefore, their images of $a_1 \tau_E + a_2 \delta$ cannot be distinguished by any rational function $f$ defined on such torsion points, which mean that we have $\widehat{f}_a(\rho, s) = \widehat{f}_a(\lambda, t)$ for any $a \in N^{-1}O_E/O_E$ and rational functions $f$ over $K$ defined for $a$. (Hence $X_N$ is contained in the galois object $\Hom_{G_K}(D_N, K^{ab})$ dual to $D_N$.)

Conversely, suppose $\widehat{f}_a(\rho, s) = \widehat{f}_a(\lambda, t)$ for any $a \in N^{-1}O_E/O_E$ and rational functions $f$ over $K$ defined for $a$; this means that each of the $N$-torsion points $a_1 \tau_E + a_2 \delta \in N^{-1}O_E/O_E$ is mapped to the points, say $p_{\rho,s}(a), p_{\lambda,t}(a) \in \proj^d(\comp)$, by the two maps in (\ref{defining diagram of congruence N}) respectively so that they cannot be distinguished by any such rational functions $f$. We shall prove that, if these points $p_{\rho,s}(a), p_{\lambda,t}(a)$ are distinct for some $a \in N^{-1}O_E/O_E$, they can be distinguished by some rational function $f$ over $K$ defined for $a$: Since we suppose $p_{\rho, s}(a) \neq p_{\lambda, t}(a) \in \proj^d(\comp)$, we can take some rational function $f_1: \proj^d(\comp) \rightarrow \proj^1(\comp)$ defined over $\ratf$ such that $f_1(p_{\rho,s}(a))\neq f_1(p_{\lambda,t}(a)) \in \proj^1(\comp)$. Then consider the finite subset $S \subseteq \proj^1(\comp)$ consisting of $f_1(p_{\rho,s}(a'))$ and $f_1(p_{\lambda,t}(a'))$ for all $a' \in N^{-1}O_E/O_E$, which (of course) contains $f_1(p_{\rho,s}(a))$ and $f_1(p_{\lambda,t}(a))$. If $S$ does not contain $\infty \in \proj^1(\comp)$, the original $f_1$ already gives our desired one; otherwise, take a rational function $h: \proj^1(\comp) \rightarrow \proj^1(\comp)$ over $\ratf$ so that all the points in $S \subseteq \proj^1(\comp)$ are mapped injectively into $\proj^1(\comp)\setminus \{\infty\}$; then $h \circ f_1: \proj^d(\comp) \rightarrow \proj^1(\comp)$ is our desired one. This completes the proof. 
\end{proof}

From this, it follows that an algebraic Witt vector $\xi \in E_K$ is modular (i.e.\ $\xi \in M_K$) if and only if the finite monoid $D_\xi$ is a quotient of $D_N$ for some $N \geq 1$. Recall that $D_\xi$ is the quotient finite monoid $D_\xi :=I_K / \sim_\xi$, where the congruence $\sim_\xi$ is defined as $\ida \sim_\xi \idb$ if and only if $\psi_\ida \xi = \psi_\idb \xi$, or in more elementary words, $\xi_{\ida \idc} = \xi_{\idb \idc}$ for any $\idc \in I_K$. That is, the congruence $\sim_\xi$ on $I_K$ concerns only the pattern of whether or not the components $\xi_{\ida \idc}, \xi_{\idb \idc}$ coincide for all $\idc$, and does not concern their actual values. The real ingredient of this dual characterization is that whether $\xi \in E_K$ is realized as a modular vector $\xi = \widehat{f}_a$ analytically can be determined only by such a pattern of the equality between the components $\xi_\ida$ of $\xi$. 

Nevertheless, the definition of the congruence $\sim_N$ depends on theta functions $\Theta_s$ and thus the finite monoid $D_N = I_K / \sim_N$ is not yet intrinsic for $K$. The subject of the rest of this section is to remedy this issue, giving a purely adelic description of the structure of $D_N$, for which the following lemma is our key result. In what follows, for each subset $S \subseteq O_E/NO_E$, let us denote by $A_E[S]$ the $N$-torsion points on $A_E$ that correspond to $S$ via the isomorphism $A_E[N] \simeq O_E/NO_E$; in particular, for each $\rho \in \widehat{O}_K$, we denote by $S_\rho \subseteq O_E/NO_E$ the image of the map $\eta(\rho): O_E/NO_E \rightarrow O_E/NO_E$.

\begin{lem}
\label{key lemma}
For any $[\rho, s], [\lambda, t] \in DR_K$, we have $[\rho, s] \sim_N [\lambda, t]$ if and only if there exists some $w \in  (O_E/NO_E)^\times$ satisfying the following two conditions:
\begin{enumerate}
\item $\eta(\rho) w = \eta(\lambda)$ in $O_E/N O_E$; 
\item the action of $[s^{-1}t] \in G_K^{ab}$ on $A_E[S_\rho]$ is equal to that of $w$. 
\end{enumerate}
where the action of $w \in (O_E/NO_E)^\times$ on $A_E[N]$ is given by the multiplication $A_E [N] \simeq O_E/NO_E \ni a \mapsto a \cdot w \in O_E/N O_E \simeq A_E[N]$. 
\end{lem}
\begin{proof}
We first prove the if part; suppose that the above two conditions hold for $[\rho, s], [\lambda, t] \in DR_K$ and $w \in (O_E/NO_E)^\times$, from which we deduce $[\rho, s] \sim_N [\lambda, t]$. Since $\eta(\rho)\cdot w = \eta(\lambda) \in O_E / NO_E$, we have the following equalities for any $k, l \in \delta^{-1}\integer^g/\integer^g$ and $a \in N^{-1}O_E/ O_E$:
\begin{eqnarray}
 \widehat{\theta}^{k/l}_a (\lambda, t) &=& \biggl(\theta^{k/l}_{a\eta(\lambda)} (\tau_E)\biggr)^{[t]^{-1}} \\
 &=& \biggl( \theta^{k/l}_{a \eta(\rho) w} (\tau_E) \biggr)^{[t]^{-1}}
\end{eqnarray}
where $\widehat{\theta}^{k/l}(\rho, s)$ is defined as in Definition \ref{defn of modular vector}. But now by $a\eta(\rho) \in S_\rho$ and by the second condition, this last one is equal to $\widehat{\theta}^{k/l}_a (\rho, s)$, thus:
\begin{eqnarray}
 \widehat{\theta}^{k/l}_a (\lambda, t) &=& \widehat{\theta}^{k/l}_a (\rho, s). 
\end{eqnarray}
Since $a \in N^{-1}O_E/O_E$ and $k, l \in \delta^{-1}\integer^g/\integer^g$ are arbitrary, this implies that $[\rho, s] \sim_N [\lambda, t]$. 

Conversely, we prove the only-if part: suppose that we have $[\rho, s] \sim_N [\lambda, t]$, from which we shall construct $w \in (O_E/NO_E)^\times$ satisfying the above two conditions. To this end, we note that the first condition almost immediately implies the second one: in fact, if we have $\eta(\rho) w = \eta(\lambda)$, the above argument shows $(\theta^{k/l}_{a\eta(\rho)} (\tau_E))^{[s^{-1}t]} = (\theta^{k/l}_{a \eta(\rho) w}(\tau_E))$ for any $k, l \in \delta^{-1}\integer^g/\integer^g$ and $a \in N^{-1}O_E/O_E$; this means that $[s^{-1}t] \in G_K^{ab}$ defines the same action on $A_E[S_\rho]$ as $w$, hence the second condition is satisfied. Therefore it suffices to prove that we have $w \in (O_E/NO_E)^\times$ satisfying the first condition.

We prove this by essentially the same idea of our argument in \S 4.3 \cite{Uramoto20}: Note that $[\rho, s] \sim_N [\lambda, t]$ implies the equality $(\eta(\rho)O_E, NO_E) = (\eta(\lambda)O_E, NO_E)$ of the greatest common divisors. Indeed, by the definition of the congruence $\sim_N$, the assumption $[\rho, s] \sim_N [\lambda, t]$ implies the following equalities for all $a \in N^{-1}O_E/O_E$ and $k, l \in \delta^{-1}\integer^g/\integer^g$:
\begin{eqnarray}   
 \theta^{k/l}_{a\eta(\rho)} (\tau_E) &=& \biggl( \theta^{k/l}_{a \eta(\lambda)} (\tau_E) \biggr)^{[st^{-1}]}. 
\end{eqnarray}
In particular, applying this equality to the case of $a=0$ and using a decomposition $\eta(st^{-1}) = u \cdot \alpha$ (cf.\ Lemma \ref{decomposition}) together with Shimura's reciprocity law, we have:
\begin{eqnarray}
 \theta^{k/l}_0 (\tau_E) &=&  \theta^{k/l}_0 (\alpha(\tau_E)), 
\end{eqnarray}
for all $k, l \in \delta^{-1}\integer^g/\integer^g$. With this in mind, suppose now that $a \eta(\lambda) = 0$ for some $a \in N^{-1}O_E/O_E$; then we can show $a \eta(\rho) = 0$ for this $a$ as well. In fact, for any $k, l \in \delta^{-1} \integer^g/\integer^g$:
\begin{eqnarray}
 \theta^{k/l}_{a \eta(\rho)} (\tau_E)  &=& \theta^{k/l}_0 (\alpha(\tau_E))\\
 &=& \theta^{k/l}_0 (\tau_E);
\end{eqnarray}
thus the injectivity of $\Theta_{\tau_E}: \comp^g/\Lambda_E \hookrightarrow \proj^d(\comp)$ implies $a \eta(\rho) = 0$. The converse argument then shows that $a\eta(\rho) = 0$ if and only if $a \eta(\lambda) = 0$ for all $a \in O_E/NO_E$, which means that $(\eta(\rho)O_E, NO_E) = (\eta(\lambda)O_E, NO_E)$ as desired, and equivalently, there exists some $w \in (O_E/ NO_E)^\times$ such that $\eta(\rho) w = \eta(\lambda)$ in $O_E/NO_E$, hence the claim. 
\end{proof}

Next we characterize, in terms of an adelic condition, when $[s^{-1}t] \in G_K^{ab}$ defines the same action on subsets  of $A_E[N]$ as that of $w \in (O_E/NO_E)^\times$, following the ideas of Theorem 2.5, Streng \cite{Streng} and Corollary 18.9, Shimura \cite{Shimura}. In what follows we suppose that our CM-type $(E, \Phi)$ is \emph{primitive}, whence $\End_\ratf(A_E) \simeq E$ (cf.\ \S 8.2, \cite{Shimura_abelian}):

\begin{lem}
\label{adelic description of galois group}
For each $S \subseteq (O_E/NO_E)$, the galois action of $[s]^{-1} \in G_K^{ab}$ on $A_E[S]$ is equal to that of $w \in (O_E/NO_E)^\times$ on $A_E[S]$ if and only if $s$ is a member of the following subset $T_{S, w} \subseteq \Ad_{K,f}^\times$:
\begin{equation}
T_{S, w} := \bigl\{ s \in \Ad_{K,f}^\times \mid \exists r \in E^\times, \hspace{0.1cm}  r r^\iota = N(s O_K), \hspace{0.1cm} r O_E = \eta(s) O_E , \hspace{0.1cm} r \cdot a \cdot w = a \cdot \eta(s) \hspace{0.1cm} (\forall a \in S)  \bigr\}. 
\end{equation}
\end{lem}
\begin{proof}
First, suppose that $s \in \Ad_{K,f}^\times$ defines the same action as $w \in (O_E/NO_E)^\times$ on $A_E[S]$. Then by definition, we have $(\theta^{k/l}_a (\tau_E))^{[s]^{-1}} = \theta^{k/l}_{aw} (\tau_E)$ for any $k, l \in \delta^{-1}\integer^g/\integer^g$ and $a \in S_\rho$. Applying this to the case of $a = 0$ in particular, we obtain: 
\begin{eqnarray}
 \theta_0^{k/l} (\alpha_s(\tau_E)) &=& \theta^{k/l}_0 (\tau_E),
\end{eqnarray}
for any $k, l \in \delta^{-1}\integer^g/\integer^g$. By Proposition 8.10.2 \cite{Birkenhake_Lange} claiming that the theta null values preserve the isomorphism classes of abelian varieties, we see that $A_{\tau_E}$ and $A_{\tau_s}$ are isomorphic. Thus we obtain an isomorphism $h: \comp^g/ \Phi(O_E) \rightarrow \comp^g/\Phi(\eta(s)O_E)$ of complex tori, preserving their polarizations, which is given by the multiplication of some $r \in E^\times$ such that $r O_E = \eta(s)O_E$ and $\Theta_E (\Phi(a')) = \Theta_s (\Phi(ra'))$ for all $a' \in E/O_E$. Now for any $a \in S \subseteq N^{-1}O_E/O_E$ in particular, the assumption on $s$ implies the identities $\Theta_s (\Phi(r a w)) = \Theta_E (\Phi(aw)) = \Theta_E (\Phi(a))^{[s]^{-1}} = \Theta_s (\Phi (a \cdot \eta(s)))$; therefore, by the injectivity of $\Theta_s: \comp^g/\Phi(\eta(s)O_E) \rightarrow A_s$, we have $r \cdot a \cdot w = a \cdot \eta(s)$. Moreover, since the isomorphism $h$ preserves the polarizations, it follows by Theorem 18.6 \cite{Shimura_abelian} that $r r^\iota = N(sO_K)$, hence $s \in T_\rho$. Conversely let $s \in T_{S, w}$ and take $r \in E^\times$ as in the definition of $T_{S, w}$. Then by definition, the multiplication by $r$ defines an isomorphism $h: \comp^g/\Phi(O_E) \rightarrow \comp^g/\Phi(\eta(s)O_E)$ of complex tori in such a manner that we have the identities $\Theta_E(\Phi(aw)) = \Theta_s(\Phi(r a w))= \Theta_s(\Phi(a \cdot \eta(s))) = \Theta_E(\Phi(a))^{[s]^{-1}}$ for each $a \in S$, which means that the galois action of $[s]^{-1}$ on $A_E[S]$ is equal to that of $w$, as required.  
\end{proof}

Now we conclude our target adelic description of the finite monoids $D_N$: 

\begin{thm}[adelic description of $D_N$]
\label{main result}
For any $[\rho, s], [\lambda, t] \in DR_K$, we have $[\rho, s] \sim_N [\lambda, t]$ if and only if there exists some $w \in (O_E/NO_E)^\times$ satisfying the following two conditions:
\begin{enumerate}
 \item $\eta(\rho) w = \eta(\lambda)$ in $N^{-1}O_E/O_E$; 
 \item $s^{-1} t \in T_{\rho, w} = T_{\lambda, w}$. 
\end{enumerate}
where $T_{\rho, w}, T_{\lambda, w} \subseteq \Ad_{K,f}^\times$ denote $T_{S_\rho, w}, T_{S_\lambda, w}$ respectively, which coincide under the first condition. 
\end{thm}
\begin{proof}
This follows from Lemma \ref{key lemma} and Lemma \ref{adelic description of galois group}. In particular, note that the first condition implies $S_\rho = S_\lambda$.
\end{proof}

\begin{rem}
When $E$ is imaginary quadratic, the above results readily imply yet another form of the modularity theorem (cf.\ Theorem 4.3.1 \cite{Uramoto20}); that is, the modular vectors constructed by ratios of theta functions also exhaust all algebraic Witt vectors as in the case of those constructed by Fricke functions. 
To see this it suffices to show that $[\rho, s] \sim_N [\lambda, t]$ if and only if $[\rho, s]=[\lambda, t] \in DR_{NO_K}$, in other words, $D_N$ is isomorphic to $DR_{NO_K}$. 

Since $E$ is imaginary quadratic, note first that its reflex field $K=E^*$ coincides with $E$ itself; also the reflex norm $\eta: \Ad_{K, f} \rightarrow \Ad_{E, f}$ is the identity. Then Theorem \ref{main result} says that $[\rho, s] \sim_N [\lambda, t]$ if and only if there exists some $u \in (O_K/NO_K)^\times$ such that $\rho u = \lambda$ in $O_K/NO_K$ and $s^{-1}t \in T_{\rho, u}$; but as in the proof of Lemma \ref{key lemma}, one can also see that $s^{-1}t \in T_{\rho, u}$ if and only if the actions of $[s]$ and $[ut]$ on the subset $A_K[S_\rho]$ of the $N$-torsion points on the elliptic curve $A_K$ coincide, or equivalently, $[s] = [ut]$ in the ray class group $C_{\idd_\rho}$ for the conductor $\idd_\rho := NO_K/ (\rho O_K, NO_K)$. That is, the finite monoid $D_N$ is isomorphic to $(O_K/NO_K) \times_{(O_K/NO_K)^\times} C_{NO_K} = DR_{NO_K}$, hence the claim. 
\end{rem}

In this sense, Theorem \ref{main result} can be seen as a natural generalization of our previous modularity theorem for algebraic Witt vectors over imaginary quadratic fields (Theorem 4.3.1 \cite{Uramoto20}) to the case of (primitive) CM fields; in particular, we showed that the use of ratios of theta functions instead of Fricke functions enables this higher-dimensional extension in a natural way, and also highlights the geometric nature of the modularity theorem in the sense that we discussed in \S \ref{s3}.



\bibliographystyle{abbrv}
\bibliography{semigalois4}

\begin{thebibliography}{10}

\bibitem{Birkenhake_Lange}
C.~Birkenhake and H.~Lange.
\newblock {\em Complex {A}belian {V}arieties}.
\newblock Springer, 2004.

\bibitem{Borger_Smit08}
J.~Borger and B.~de~Smit.
\newblock Galois theory and integral models of {$\Lambda$}-rings.
\newblock {\em Bulletin of London Mathematical Society}, 40(3):439--446, 2008.

\bibitem{Borger_Smit11}
J.~Borger and B.~de~Smit.
\newblock Lambda actions of rings of integers, 2011.
\newblock arXiv:1105.4662.

\bibitem{Deligne_Ribet}
P.~Deligne and K.~Ribet.
\newblock Values of {A}belian {L}-functions at negative integers over totally
  real fields.
\newblock {\em Inventiones mathematicae}, 59:227--286, 1980.

\bibitem{Ha_Paugam}
E.~Ha and F.~Paugam.
\newblock {B}ost-{C}onnes-{M}arcolli systems for {S}himura varieties. {I}.
  {D}efinitions and formal analytic proverties.
\newblock {\em Int. Math. Res. Pap.}, 5:237--286, 2005.

\bibitem{scss2021}
Y.~Ishitsuka and T.~Uramoto.
\newblock On the integrality of algebraic {W}itt vectors over imaginary
  quadratic fields.
\newblock In {\em 9th International Symposium on Symbolic Computation in
  Software Science (SCSS 2021) short and work-in-progress papers}, pages
  28--33, 2021.

\bibitem{Johansson_Enge_Hart}
F.~Johansson, A.~Enge, and W.~Hart.
\newblock Addition sequences and numerical evaluation of modular forms, 2015.
\newblock https://fredrikj.net/math/strobl2015.pdf.

\bibitem{Mazur}
B.~Mazur.
\newblock A brief introduction to the work of {H}ida, 2012.

\bibitem{Milne}
J.~Milne.
\newblock Introduction to {S}himura varieties.
\newblock In {\em Harmonic analysis, the trace formula, and Shimura varieties},
  Clay Math. Proc. 4, pages 265--378. American Mathematical Society, 2004.

\bibitem{Milne_Shih}
J.~Milne and K.-y. Shih.
\newblock Automorphism groups of {S}himura varieties and reciprocity laws.
\newblock {\em American Journal of Mathematics}, 103:911--935, 1981.

\bibitem{Mumford}
D.~Mumford.
\newblock {\em Abelian varieties}.
\newblock American Mathematical Society, 2012.

\bibitem{Shimura}
G.~Shimura.
\newblock {\em Introduction to the arithmetic theory of automorphic functions}.
\newblock Princeton University Press, 1971.

\bibitem{Shimura_abelian}
G.~Shimura.
\newblock {\em Abelian varieties with complex multiplication and modular
  functions}.
\newblock Princeton University Press, 1997.

\bibitem{Shimura_arithmeticity}
G.~Shimura.
\newblock {\em Arithmeticity in the theory of automorphic forms}.
\newblock American Mathematical Society, 2000.

\bibitem{Streng}
M.~Streng.
\newblock An explicit version of {S}himura's reciprocity law for {S}iegel
  modular functions, 2021.
\newblock arXiv:1201.0020v3.

\bibitem{Uramoto16}
T.~Uramoto.
\newblock Semi-galois {C}ategories {I}: {T}he {C}lassical {E}ilenberg {V}ariety
  {T}heory.
\newblock In {\em Proceedings of the 31st Annual ACM/IEEE Symposium on Logic in
  Computer Science (LICS'16)}, pages 545--554, 2016.

\bibitem{Uramoto17}
T.~Uramoto.
\newblock Semi-galois {C}ategories {I}: {T}he {C}lassical {E}ilenberg {V}ariety
  {T}heory, 2017.
\newblock arXiv:1512.04389v4 (Extended version of LICS'16 paper).

\bibitem{Uramoto18}
T.~Uramoto.
\newblock Semi-galois {C}ategories {II}: {A}n arithmetic analogue of
  {C}hristol's theorem.
\newblock {\em J. Algebra}, 508(2018):539--568, 2018.

\bibitem{Uramoto20}
T.~Uramoto.
\newblock Semi-galois {C}ategories {III}: {W}itt vectors by deformations of
  modular functions, 2021.
\newblock arXiv:2007.13367v7.

\bibitem{Yalkinoglu12}
B.~Yalkinoglu.
\newblock On {B}ost-{C}onnes type systems and complex multiplication.
\newblock {\em Journal of Noncommutative Geometry}, pages 275--319, 2012.

\bibitem{Yalkinoglu13}
B.~Yalkinoglu.
\newblock On arithmetic models and functoriality of {B}ost-{C}onnes systems.
  with an appendix by {S}ergey {N}eshveyev.
\newblock {\em Invent. Math.}, 191(2):383--425, 2013.

\end{thebibliography}
\end{document}